\pdfoutput=1
\documentclass[twoside,a4paper]{siamltex}

\usepackage{amsmath}
\usepackage{soul}
\usepackage{mathtools}
\usepackage{amssymb}
\usepackage{amsfonts}
\usepackage{amsxtra}
\usepackage{amstext}
\usepackage{amsbsy}
\usepackage{amscd}
\usepackage{graphicx}
\usepackage{float}
\usepackage{cite}
\usepackage{xcolor}
\usepackage{srcltx} 
\usepackage{marginnote}
\usepackage{rotating}
\usepackage{booktabs} 
\usepackage{url} 
\usepackage{algorithm,algorithmicx,algpseudocode} 
\usepackage{comment}
\usepackage{calc}
\usepackage{multirow}
\usepackage{nicefrac}
\usepackage{siunitx}  
\usepackage{subcaption}
\usepackage{enumerate}
\usepackage{longtable}
\usepackage{mathabx}
\usepackage[font=footnotesize,labelfont=bf]{caption}

\newcommand{\mcJ}{\mathcal{J}}

\newcommand{\mcC}{\mathcal{C}}
\newcommand{\mcD}{\mathcal{D}}
\newcommand{\mcG}{\mathcal{G}}

\newcommand{\mcL}{\mathcal{L}}

\newcommand{\mcX}{\mathcal{X}}

\newcommand{\mbR}{\mathbb{R}}

\newcommand{\mbRd}{{\mathbb{R}^d}}

\newcommand{\Uad}{\mathcal{U}_{ad}}

\newcommand{\Px}[2]{ \mathbf{P}_{#1} (#2)}
\newcommand{\px}[2]{ \mathbf{p}_{#1} (#2)}
\newcommand{\pic}[4]{  \subcaptionbox{SSIM: #1}
		{\includegraphics[width=#2\textwidth]{Results/#3/#4.png}}}
\newcommand{\picf}[5]{ \subcaptionbox{SSIM: #1}
		{\includegraphics[width=#2\textwidth]{Results/#3/#4/#5.png}}}
\newcommand{\picl}[4]{ 
		{\includegraphics[width=#1\textwidth]{Results/#2/#3/#4.png}}}

\newcommand{\pica}[4]{ \subcaptionbox{#1}
		{\includegraphics[width=#2\textwidth]{Results/#3/#4.png}}}
\newcommand{\picb}[4]{ \subcaptionbox*{#1}
		{\includegraphics[width=#2\textwidth]{Results/#3/#4.png}}}

\newcommand{\picc}[2]{ \subcaptionbox*{#1}
		{\hspace{#2\textwidth}}}

\def \Omg{{\Omega}}
\def \Omgi{{\Omg_I}}
\def \Omgf{{\Omg\cup\Omg_I}}  
\def \oOmg{{\overline\Omg}}

\def \alphab{{\boldsymbol\alpha}}
\def \taub{{\boldsymbol\tau}}
\def \lambdab{{\boldsymbol \lambda}}

\def \nub{{\boldsymbol \nu}}

\def \fb{\mathbf{f}}
\def \ub{\mathbf{u}}
\def \wb{\mathbf{w}}
\def \vb{\mathbf{v}}
\def \xb{\mathbf{x}}

\def \pb{\mathbf{p}}

\def \tb{\mathbf{t}}

\def \yb{\mathbf{y}}

\def \zerob{\mathbf{0}}
\def \ib{\boldsymbol{i}}
\def \jb{\boldsymbol{j}}

\def \etab{{\boldsymbol\eta}}

\def \veps{\varepsilon}
\def \wa{{w^*}}
\def \ua{{u^*}}
\def \za{{z^*}}

\newcommand{\weak}{\rightharpoonup}
\newtheorem{remark}{Remark}
\newcommand{\func}[3]{%
#1\colon  #2  \longrightarrow~#3}

\definecolor{blue(ncs)}{rgb}{0.0, 0.53, 0.74}

\begin{document}
\title{Bilevel parameter learning for nonlocal image denoising models}
\author{M. D'Elia\footnotemark[3] \and J.C. De los Reyes\footnotemark[1] \,\footnotemark[4] \and A. Miniguano-Trujillo\footnotemark[4]}
\renewcommand{\thefootnote}{\fnsymbol{footnote}}
\footnotetext[1]{Corresponding author, email: \url{juan.delosreyes@epn.edu.ec}}
\footnotetext[3]{Computational Science and Analysis, Sandia National Laboratories, CA, USA.}
\footnotetext[4]{Research Center on Mathematical Modelling (MODEMAT), Escuela Polit{\'e}cnica Nacional, Quito, Ecuador}
\renewcommand{\thefootnote}{\arabic{footnote}}

\maketitle

\begin{abstract}
We propose a bilevel optimization approach for the estimation of parameters in nonlocal image denoising models. The parameters we consider are both the fidelity weight and weights within the kernel of the nonlocal operator. In both cases we investigate the differentiability of the solution operator in function spaces and derive a first order optimality system that characterizes local minima. For the numerical solution of the problems, we use a second-order trust-region algorithm in combination with a finite element discretization of the nonlocal denoising models and we introduce a computational strategy for the solution of the resulting dense linear systems. Several experiments illustrate the applicability and effectiveness of our approach.
\end{abstract}

\section{Introduction}\label{sec:introduction}
Nonlocal image denoising has emerged in the last years as an important alternative in image processing, due to the fact that it enables the reconstruction of important image features, specially textures, by considering similar intensity patterns between pixels or patches in a given spatial neighbourhood or all over the whole image domain. Although originally the research was focused on the design of direct nonlocal filters (see, e.g., \cite{yaroslavsky1986digital,smith1997susan,tomasi1998bilateral}), more complex variational approaches based on energy functionals were proposed in recent years for the treatment of denoising problems \cite{gilboa2008nonlocal,gilboa2007nonlocal,lou2010image}. This variational framework enables the employment of additional modeling and analysis tools that have been used for image reconstruction tasks within a partial differential equation (PDE) setting. A related variational framework was also employed recently in \cite{Antil2017imagingSpectral}, where the authors use energies induced by (nonlocal) fractional differential operators. 

Nonlocal denoising operators are characterized by kernels that incorporate a lot of information from the image domain. The use of different kernels leads in fact to different outcomes, and tuning their parameters is usually a difficult task. In recent years bilevel optimization has been successfully utilized for the identification of optimal parameters in image processing \cite{de2013image,de2017bilevel,kunisch2013bilevel}. This attempt includes analytical as well as numerical studies, using both finite-dimensional \cite{kunisch2013bilevel,hintermuller2015bilevel} and PDE-constrained optimization approaches \cite{de2013image,de2017bilevel,hintermuller2017optimal}. 

In this paper we aim at extending the bilevel optimization methodology to nonlocal operators with integrable kernels used for image denoising. Similar to previous contributions, we consider a supervised learning framework and assume existence of a training set of clean and noisy images we can learn from. Using a variational setting similar to the one developed in \cite{d2014optimal,d2016identification}, {we analyze the resulting bilevel learning problems in function spaces. Differentiability properties of the solution mappings are investigated and necessary optimality conditions of Karush-Kuhn-Tucker type are derived. We consider both spatially-dependent functional weights and also scalar ones. However, the framework is also extendable to parameter functions that depend on a finite, yet possible large, number of coefficients, which is the most common approach in practice to avoid overfitting.}

To our knowledge, this is the first paper on bilevel optimization for nonlocal operators. In particular, the second part of the paper addresses the problem of nonlocal kernel identification, now subject of great interest in the nonlocal community, and provides an alternative to neural-networks-based algorithms \cite{pang2019fpinns}. As such, the impact of this work goes beyond image processing, providing a useful tool in the context of nonlocal optimization and control for a wide range of applications including fracture mechanics \cite{Ha2011,Littlewood2010,Silling2000},
anomalous subsurface transport \cite{Benson2000,Schumer2003,Schumer2001},
phase transitions \cite{Bates1999,Delgoshaie2015,Fife2003},
multiscale and multiphysics systems \cite{Alali2012,Askari2008},
magnetohydrodynamics \cite{Schekochihin2008},
and stochastic processes \cite{Burch2014,DElia2017,Meerschaert2012,MeKl00}. 

The paper is organized as follows. In Section \ref{sec:nonlocal} we briefly summarize results of the nonlocal vector calculus that will be useful throughout the paper and also introduce some nonlocal operators for image denoising. In Section \ref{sec:opt-wrt-lambda} we consider a bilevel optimization approach to optimize the fidelity weight, both spatially-dependent and scalar, for a Gaussian denoising problem and derive necessary optimality conditions in form of Karush-Kuhn-Tucker optimality systems. In Section \ref{sec:opt-wrt-w}, we introduce and analyze the bilevel problem of finding optimal weights of a modified \emph{nonlocal means} kernel and, after studying some properties of the solution mapping, we derive necessary  optimality conditions of first order. Finally, in Section \ref{sec:numerical-tests}, we introduce a second-order optimization algorithm for the solution of the bilevel problems and give insights of implementation aspects and numerical performance. Several numerical tests illustrate the performance of our approach.

\section{Preliminaries in nonlocal imaging} \label{sec:nonlocal}
Let $\Omg$ be a bounded domain in $\mbRd$. We use the standard notation $(\cdot,\cdot)_{0,\Omg}$ and $\|\cdot\|_{0,\Omg}$ for the inner product and the norm in $L^2(\Omg)$, the space of square integrable functions on $\Omg$.

\subsection{Nonlocal vector calculus}\label{sec:nlvc}
The nonlocal models considered in this paper are analyzed using the nonlocal vector calculus \cite{Du_12_SIREV}. We recall the basic concepts of such calculus that will be used in this paper. {We let $\alphab(\xb,\yb)=-\alphab(\yb,\xb):\mbRd\times\mbRd\to\mbRd$ be an anti-symmetric vector function and $\gamma(\xb,\yb)=\gamma(\yb,\xb):\mbRd\times\mbRd\to\mbR$ be a symmetric positive kernel, square integrable over $\Omg$. Given the functions $u(\xb):\mbRd\to\mbR$ and $\nub(\xb,\yb):\mbRd\times\mbRd\to\mbRd$, we} define the nonlocal divergence of $\nub$ as a mapping $\mcD\nub:\mbRd\to\mbR$ such that
\begin{equation}\label{eq:ndiv}
	\mcD\nub(\xb) := \int_{\mbRd} \big(\nub(\xb,\yb)+\nub(\yb,\xb)\big)\cdot\alphab(\xb,\yb)\,d\yb,\qquad \xb\in\mbRd,
\end{equation}
and the nonlocal gradient of $u$ as a mapping $\mcG u:\mbRd\times\mbRd\to\mbR$ such that
\begin{equation}\label{eq:ngra}
	\mcG u(\xb,\yb) := \big(u(\yb)-u(\xb)\big)  \alphab(\xb,\yb), \quad \forall \;\xb,\yb\in\mbRd.
\end{equation}
The paper \cite[\S3.2]{Du_12_SIREV} shows that the adjoint $\mcD^*=-\mcG$, as in the local case. The composition of nonlocal divergence and gradient gives
\begin{displaymath}
	\mcD\big(\mcG u)(\xb)  = 2\int_{\mbRd}\big(u(\yb)-u(\xb)\big) \big( \alphab(\xb,\yb)\cdot\alphab(\xb,\yb)\big) \,d\yb.
\end{displaymath}
With the identification $\gamma(\xb,\yb):=\alphab(\xb,\yb)\cdot\alphab(\xb,\yb)$ we define the nonlocal diffusion of $u$ as the operator $\mcL u :\mbRd\to\mbR$ such that
\begin{displaymath}
	\mcL u(\xb) := \mcD\big(\mcG u)(\xb)=
	2\int_\mbRd \big(u(\yb)-u(\xb)\big) \, \gamma (\xb, \yb )\,d\yb, \quad \forall \;\xb \in \mbRd.
\end{displaymath}
Then, we define the interaction domain $\Omgi$ of a bounded region $\Omg$ as the set of points outside of the domain that interact with points inside of the domain, i.e.
\begin{displaymath}
	\Omgi = \{\yb\in\mbRd\setminus\Omg: \; \gamma(\xb,\yb)\neq 0, \; \textrm{ for some}\,\xb\in \Omg\}.
\end{displaymath}
This set is the nonlocal counterpart of the boundary $\partial\Omg$ of a domain in a local setting. In this work we consider localized kernels, i.e. $\gamma$ is such that for $\xb\in\Omg$
\begin{equation}\label{eq:gamma-cond}
	\left\{\begin{aligned}
	\gamma(\xb,\yb)  \geq  0 \quad &\forall\, \yb\in B_\veps(\xb),
	\\
	\gamma(\xb,\yb)  = 0 \quad & \forall\, \yb\in \mbRd \setminus B_\veps(\xb),
	\end{aligned}\right.
\end{equation}
where $B_\veps(\xb) = \{\yb\in \mbRd: \; |\xb-\yb|_\infty\leq\veps\}$, for all $\xb\in\Omg$ and $\veps>0$ is referred to as interaction radius\footnote{Note that, in general, nonlocal neighborhoods are Euclidean balls. However, the nonlocal calculus still holds for more general balls such as those induced by the $\ell$-infinity norm (see an application in \cite{Capodaglio2019}).}. For such kernels, we can rewrite the interaction domain as
\begin{displaymath}
	\Omgi = \{ \yb\in \mbRd\setminus\Omg: \; |\yb-\xb|_\infty\leq\veps,\;\textrm{ for some}\,\xb\in\Omg\}.
\end{displaymath}
We define the nonlocal energy semi-norm, the nonlocal energy space and the constrained nonlocal energy space as follows
\begin{equation}\label{eq:norm-space-def}
\begin{array}{l}
\|v\|^2_V  := \displaystyle \int_\Omgf \int_\Omgf \big( v(\xb)-v(\yb) \big)^2 \gamma(\xb,\yb)\,d\yb d\xb,
	             \quad \forall \, v\in L^2(\Omg), \\[4mm]
V(\Omgf)   := \{v\in L^2(\Omgf): \|v\|_V<\infty\},\\[3mm]
V_c(\Omgf) := \{v\in V(\Omgf): v|_{\Omega_I}=0\}.
	\end{array}
\end{equation}
The paper \cite[\S4.3.2]{Du_12_SIREV} proves that for integrable localized kernels as in \eqref{eq:gamma-cond} the constrained energy space $V_c(\Omgf)$ is equivalent to $$
L^2_c(\Omgf):=\{v\in L^2(\Omgf):v|_{\Omega_I}=0\}
$$
and that $\|\cdot\|_V\sim\|\cdot\|_{L^2(\Omgf)}$. Unless necessary, we drop the dependence of $V$ and $V_c$ on $\Omgf$.

{\bf Nonlocal volume constrained problems} We consider the solution of nonlocal elliptic problems, i.e., the nonlocal counterpart of elliptic PDEs. Due to nonlocality, when solving a nonlocal problem, boundary conditions (i.e. conditions on the solution for $\xb\in\partial\Omega$) do not guarantee the uniqueness of the solution, which can only be achieved by providing conditions on the interaction domain $\Omega_I$ \cite{Du_12_SIREV}. As an illustrative example, we consider the following {\it nonlocal diffusion-reaction equation} for the scalar function $u$:
\begin{align}\label{eq:diffusion-reaction}
-\mcL u + \lambda u &= f  \qquad \xb\in\Omega,
%
\intertext{for some $f\in L^2(\Omega)$ and $\lambda\in L^\infty(\Omega)$ such that $\lambda:\Omega\to\mathbb R^+ \cup \{0\}$. Uniqueness of $u$ is guaranteed provided the following condition is satisfied \cite{Du_12_SIREV}:}
%
\label{eq:nBC}
u &= g  \qquad \text{ for }\xb\in\Omega_I,
\end{align}
where $g$ is some known function in the {\it trace space}
$$
\widetilde V(\Omega_I)=\{z:\exists \; v\in V\;\; {\rm s.t.} \;\; v|_{\Omega_I}=z\}.
$$
Without loss of generality, in our analysis we consider $g=0$ so that $u\in V_c(\Omega\cup\Omega_I)$.
The corresponding weak form is obtained in the same way as in the local setting by multiplying \eqref{eq:diffusion-reaction} by a test function and integrating over $\Omega$, i.e.
\begin{equation}\label{eq:weak-diffusion-reaction}
\int_\Omega (-\mcL u +\lambda u - f)v\,d\xb =
\int_{\Omega\cup\Omega_I}\int_{\Omega\cup\Omega_I} \mcG u \,\mcG v \,d\xb\,d\yb
+ \int_\Omega (\lambda u - f)v\,d\xb= 0,
\end{equation}
where the equality follows from the nonlocal Green's {first} identity \cite{Du_12_SIREV}. Note that, by definition of $\mcG$, \eqref{eq:weak-diffusion-reaction} is equivalent to
\begin{equation}\label{eq:weak-expl-diffusion-reaction}
\int_{\Omega\cup\Omega_I}\int_{\Omega\cup\Omega_I} (u(x)-u(y))(v(x)-v(y))\gamma(\xb,\yb)\,d\xb\,d\yb
+ \int_\Omega (\lambda u - f)v\,d\xb = 0.
\end{equation}

\subsection{Nonlocal denoising formulation}\label{sec:imaging-problem}
In order to use the nonlocal vector calculus for image denoising models, we consider the variational viewpoint proposed in \cite{gilboa2007nonlocal}, applicable to the following kernels:
\begin{itemize}
    \setlength{\itemindent}{-.2in}
    \setlength{\itemsep}{1em}
  \item Yaroslavsky kernel:
  $
  \gamma_1(\xb,\yb)= \exp\left\{-w \big(f(\xb)-f(\yb)\big)^2\right\}\,\mcX(\yb\in B_\veps(\xb)),$
  \item Nonlocal Means kernel:
  \begin{equation} \label{model: nonlocal means}
  \gamma_2(\xb,\yb)= \exp\left\{- \int_{\mbR^2} w(\taub)\big(f(\xb+\taub)-f(\yb+\taub)\big)^2\,d\taub\right\}\,\mcX\big(\yb\in B_\veps(\xb)\big),
  \end{equation}
  \item Combination of the previous two kernels:
  \begin{multline}\label{eq:w-combination}
  \gamma_C(\xb,\yb) = \exp\left\{-\theta _1 \big(f(\xb)-f(\yb)\big)^2 \right.\\
  \left. -\theta_2 \int_{\mbR^2} w(\taub)\big(f(\xb+\taub)-f(\yb+\taub)\big)^2\,d\taub\right\}\,\mcX\big(\yb\in B_\veps(\xb)\big),
  \end{multline}
\end{itemize}
where $f$ is a given noisy image and $\mcX(\cdot\in B)$ is the indicator function over the set $B$. In \cite{buades2011self} it is shown that nonlocal means presents advantages in presence of textures or periodic structures, whereas neighborhood filters, like the Yaroslavsky one, may perform better for the preservation of particular edges. As a consequence, a kernel that considers a combination of both, as in \eqref{eq:w-combination}, may provide an increased denoising capability. We refer to \cite{buades2005review} for more details on these and other nonlocal kernels.

For a given kernel function, we formulate the nonlocal denoising problem as the following energy minimization problem:
\begin{equation}\label{eq:denoising-problem}
	\min_{u\in V_c} \quad  \frac{\mu}{2} \|u\|_V^2 +\frac12 \int_\Omg \lambda \, (u-f)^2 \,d\xb,
\end{equation}
where $f \in L^\infty(\Omega)$ stands for the noisy image and $\lambda$ is a weight that balances the fidelity term against the nonlocal regularizer. The weight $\lambda$ can be either a {non-negative} real number or a spatially dependent quantity. {The last term in \eqref{eq:denoising-problem} corresponds to the gaussian fidelity, which is considered for simplicity along the paper. However, the approach may be easily extended to other convex differentiable fidelity functions.}

{Depending on the chosen kernel (and the value of $w$ within) and the value of $\lambda$, reconstructed images of different quality are obtained as solutions to \eqref{eq:denoising-problem}. As an example, for a given pair $(u_T,f)$ of noisy and ground-truth images and using the nonlocal means kernel \eqref{model: nonlocal means} in problem \eqref{eq:denoising-problem}, we show in Figure \ref{fig:Level-curves} the contour lines of the loss function $\frac{1}{2} \|u_{\lambda,w}- u_T\|^2_{0,\Omega}$ associated with a scalar fidelity weight $\lambda$ and a scalar kernel weight $w$}. This two-dimensional plot exemplifies the difficulties related to the optimization. In fact, the complex banana-shaped contour lines are a challenge for several minimization algorithms, especially first-order ones.
\begin{figure}[H]
 \centering
 \includegraphics[width=0.8\textwidth]{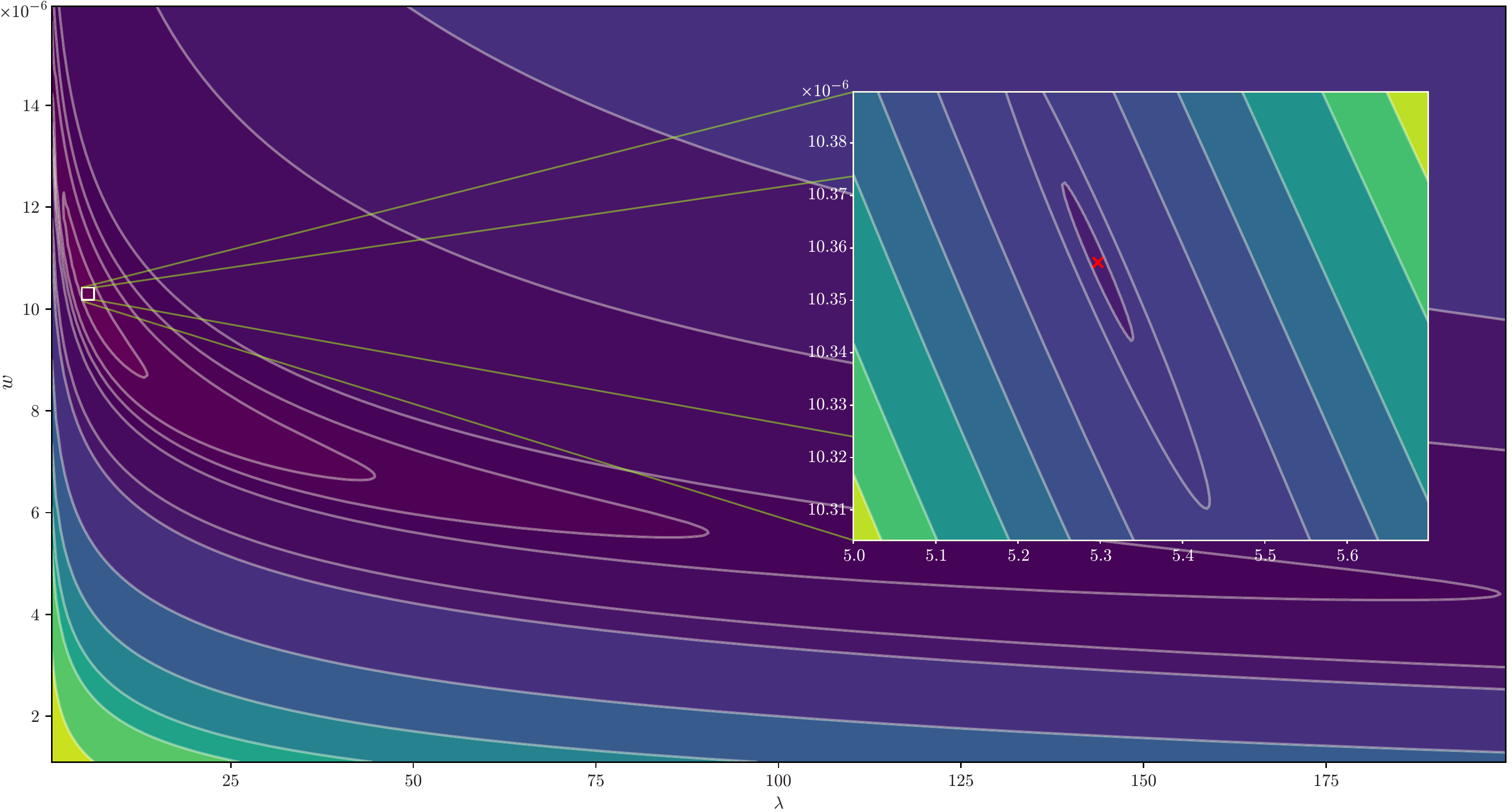}
 \\[-0.3em]
 \caption{Contour plot of the {$L^2$-squared} loss function for different values of the parameters $\lambda$ and $w$, using the nonlocal means kernel.}
 \label{fig:Level-curves}
\end{figure}

\section{Optimization with respect to $\lambda$}\label{sec:opt-wrt-lambda}
 We study the problem of identifying the optimal spatially dependent $\lambda$ in the lower-level denoising model \eqref{eq:denoising-problem}. First, we analyze the existence of a solution to the lower-level problem with fixed parameters. Then, we state the bilevel problem for the identification of the optimal $\lambda$ and study the differentiability of the solution operator and the reduced cost functional. We also derive a first order optimality system for the estimation of the optimal parameter. The case $\lambda\in\mathbb R^+ \cup \{0\}$ is studied at the end of the section as a particular instance.

\subsection{Lower-level problem}\label{sec:lower-level-lambda}
We recall the energy-based formulation of the nonlocal denoising problem:
\begin{equation}\label{eq:lower-level}
	\min_{u\in V_c} \mathcal E(u,\lambda),
\end{equation}
where
\begin{equation*}
	 \mathcal E(u,\lambda)= \frac{\mu}{2} \|u\|_V^2 +\frac12 \int_\Omg \lambda (u-f)^2 \,d\xb.
\end{equation*}
The energy norm \(\|\cdot\|_V\) is defined as in \eqref{eq:norm-space-def} and, in particular, is induced by the scalar product
\begin{equation}\label{ec:non_local_product}
\begin{aligned}
(u,v)_V  
& = \int_\Omgf \int_\Omgf \big( u(\xb)-u(\yb) \big) 
    \big( v(\xb)-v(\yb) \big) \gamma(\xb,\yb)\,d\yb d\xb \\
& = 2 \int_\Omg \int_\Omgf \big( u(\xb)-u(\yb) \big) v(\xb) 
    \gamma(\xb,\yb)\,d\yb d\xb, \quad \forall \, u,v\in V_c.
\end{aligned}
\end{equation}
In what follows we refer to \eqref{eq:lower-level} as the {\it lower-level problem} and we study its well-posedness as well as a necessary and sufficient conditions for the characterization of its minima.
\begin{theorem}\label{thm:lower-level-wellp}
	For every $\lambda \in L^\infty(\Omega)$, such that $ \lambda(\xb) \geq 0$ a.e., there exists a unique solution $u \in V_c$ for the lower-level problem \eqref{eq:lower-level}.
\end{theorem}
\begin{proof}
Since the functional $\mathcal E$ is bounded from below, there exists a minimizing sequence $\{u_n\}\subset V_c$. Thanks to the coercivity in $V_c$ of the energy term, the sequence is bounded in $V_c$; thus, there exists a subsequence, still denoted by $\{u_n\}$, that weakly converges in $V_c$, i.e., $u_n\weak u^*$. Since $\mathcal E$ is convex and continuous with respect to the energy norm, it is weak lower semi-continuous. Therefore,
$$\mathcal E(u^*) \leq \lim\inf_{n \to \infty} \mathcal E(u_n).$$
The uniqueness of the solution follows from the strict convexity of the functional.
\end{proof}

\subsection{Bilevel problem}\label{sec:bilevel-lambda}
We consider the following bilevel optimization problem
\begin{equation}\label{eq:lambda-opt}
\begin{aligned}
	\min_{\lambda\in\mcC}
	& \quad \mcJ(u,\lambda)= \ell(u)+ \frac{\beta}{2} \|\lambda\|^2_{H^1(\Omega)}\\[3mm] \hbox{s.t.}
	& \quad u = \arg\min_{u\in V_c}\;
	\left\{ \frac{\mu}{2}\|u\|^2_V +\int_\Omg \lambda (u-f)^2\,d\xb,\right \}
\end{aligned}
\end{equation}
where $\mcC = \big\{\lambda\in H^1(\Omega):0 \leq \lambda(\xb) \leq  \Lambda \big\}$, with $\Lambda \in \mathbb R^+$, {is a subset of the parameter space $\mathcal U:= H^1(\Omega) \cap L^\infty(\Omega)$. For any $\lambda \in \mathcal U$ such that $\lambda(\xb) \geq 0$ a.e. in $\Omega$, the denoising functional of the lower-level problem is strictly convex and its minimum is uniquely characterized by its first-order necessary and sufficient optimality condition, given by the nonlocal variational equation
\begin{equation}\label{eq:opc_denoising}
	\mu (u,\psi)_V + \big( \lambda(u-f),\psi \big)_{0,\Omg} =0,  \quad  \forall \psi\in V_c.
\end{equation}
By choosing in particular the test function $\psi=u$ in \eqref{eq:opc_denoising} we then get
\begin{align*}
	\mu \|u\|^2_V \leq  \mu \|u\|^2_V + \int_\Omg \lambda u^2 &{}=\int_\Omg \lambda f u \leq \|\lambda\|_{L^2} \|f\|_{\infty} \|u\|_{0,\Omg}.
\end{align*}
Using the equivalence of the energy norm and the $L^2$-norm, we then obtain the \emph{a-priori} estimate
\begin{equation}\label{eq:stability-lower}
	\|u\|_{0,\Omg}\leq K\|\lambda\|_{H^1(\Omega)},
\end{equation}
which will be used in the subsequent analysis of the differentiability of the solution operator.}

{Replacing the constraint in \eqref{eq:lambda-opt} with equation \eqref{eq:opc_denoising} then yields the following nonlocal-constrained optimization problem:
\begin{subequations}\label{eq:bilevel_reg_prob}
\begin{equation}\label{eq:opt_prob}
\min_{\lambda \in \mathcal C} \; \ell(u) + \frac{\beta}{2} \|\lambda\|^2_{H^1(\Omega)}
\end{equation}
\begin{equation}\label{eq:rest_denoising}
{\rm s.t.} \quad \mu (u,\psi)_V + \big(\lambda(u-f),\psi\big)_{0,\Omg}=0, \quad \forall \psi \in V_c.
\end{equation}
\end{subequations}
Well-posedness of \eqref{eq:rest_denoising} follows directly from Theorem \ref{thm:lower-level-wellp}, where the coercivity of the bilinear form $(\cdot,\cdot)_V$ plays a key role.}

Hereafter, the loss function $\ell(u)$ is assumed to be strictly convex and continuous with respect to $u$. The simplest case corresponds to the Peak Signal-to-Noise Ratio-related
loss function $\ell(u):=\frac12 \|u-u_T\|^2_{0,\Omg},$ which arises from a supervised learning framework, where $u_T$ corresponds to the ground truth image and $f$ to the corrupted one. In such framework, the training set is typically large (i.e. we assume several pairs $(u_T,f)$ are available) and the number of lower-level problems increases accordingly, but analytical difficulties remain the same. For this reason we restrict our attention to a single training pair $(u_T,f)$, which corresponds to a single lower-level problem. Alternative loss functions based on the image statistics have also been recently proposed \cite{hintermuller2017optimal} and may also been considered within our analytical framework.

\smallskip
\begin{theorem}
The bilevel optimization problem \eqref{eq:lambda-opt} admits a solution $\lambda^* \in\mcC$.
\end{theorem}
\begin{proof}
Since the functional $\mcJ$ is bounded from below, there exists a minimizing sequence $\{\lambda_n\}\subset\mcC$. Also, the Tikhonov term guarantees that this sequence is bounded in $H^1(\Omega)$. Thus, there exists a subsequence, still denoted by $\{\lambda_n\}$, that converges weakly in $H^1(\Omega)$ and strongly in $L^2(\Omega)$ {to a limit point $\lambda^* \in \mcC$.}

Let $u_n\in V_c$ be the unique (see Theorem \ref{thm:lower-level-wellp}) optimal solution to the lower-level problem \eqref{eq:lower-level} corresponding to $\lambda_n$. From the stability estimate \eqref{eq:stability-lower} we get that
\begin{displaymath}
\|u_n\|_{0,\Omg}\leq K\|\lambda_n\|_{H^1(\Omega)} \leq \overline K,
\end{displaymath}
and, therefore, $\{u_n\}$ is uniformly bounded in $V_c$. Thus, there exists a subsequence, that we still denote by $\{u_n\}$, that weakly converges in $V_c$ (and $L_c^2$ because of the equivalence of spaces) {to a limit point $u^*$.}
Next, we will show that $u^*=u(\lambda^*)$, i.e. the limit of $\{u_n\}$ is the solution of the lower-level problem corresponding to $\lambda^*$. 

{Indeed, thanks to the linearity and continuity of the bilinear form and the strong convergence of $\lambda_n \to \lambda^*$ in $L^2(\Omega)$, we may pass to the limit in 
\begin{equation*}
 \mu (u_n,\psi)_V + \big(\lambda_n (u_n-f),\psi\big)_{0,\Omg}=0, \quad \forall \psi \in V_c,
\end{equation*}
and get that
\begin{equation*}
 \mu (u^*,\psi)_V + \big(\lambda^* (u^*-f),\psi\big)_{0,\Omg}=0, \quad \forall \psi \in V_c.
\end{equation*}
}

{Thanks to the assumed properties of the loss function, we get that $\mcJ$ is weakly lower semicontinuous and, consequently, $\mcJ(u^*,\lambda^*) \leq \lim\inf\limits_{n\to\infty} \mcJ(u_n,\lambda_n),$ which finishes the proof.}
\end{proof}

\subsection{Differentiability of the Solution Operator}
In this section we analyze the differentiability properties of the solution mapping. The presence of the weak norm in the denoising model, does not allow us to obtain Fr\'echet differentiability results. Fortunately, first-order optimality conditions only require G\^ateaux differentiability, which is proved in the following theorem.
\begin{theorem} \label{thm: diff solution operator lambda}
	Let $\mathcal V \in L^\infty(\Omega)$ be an $\epsilon$-neighbourhood containing $\mathcal C$ and $\func{S_\lambda}{\mathcal V}{V_C}$ be the solution operator, which assigns to each $\lambda$ the corresponding solution $u \in V_C$ to \eqref{eq:rest_denoising}. Then, $S_\lambda$ is locally Lipschitz continuous and G\^ateaux differentiable, and its directional derivative $z=S_\lambda' h$ at $\lambda$ in direction $h \in \mathcal U$ is given by the unique solution of
	\begin{equation} \label{eq: linearized equation lambda}
    \int_\Omgf \int_\Omgf \mu \, \mcG z\, \mcG \psi+ \int_\Omg\lambda\,z\,\psi 
    = {-\int_\Omg} h(u-f)\, \psi, \quad \forall \psi \in V_c.
    \end{equation}
\end{theorem}
\begin{proof}
{To prove the Lipschitz continuity of the solution operator, we consider two parameters $\lambda_1,\lambda_2 \in \mathcal V$ and the corresponding solutions to equation \eqref{eq:opc_denoising}, $u_1$ and $u_2$. For $\epsilon$ small enough, equation \eqref{eq:rest_denoising} is indeed well-posed thanks to the coercivity of the bilinear form. Taking the difference of both equations we get
\begin{equation*}
 \mu (u_1-u_2,\psi)_V + \big(\lambda_1 (u_1-f),\psi\big)_{0,\Omg}- \big(\lambda_2 (u_2-f),\psi\big)_{0,\Omg}=0, \quad \forall \psi \in V_c.
\end{equation*}
Adding and subtracting the term $\big(\lambda_2 (u_1-f),\psi\big)_{0,\Omg}$ and choosing the test function $\psi = u_1-u_2$, we get that
\begin{equation*}
 \mu \|u_1-u_2\|^2_V + \big(\lambda_2 (u_1-u_2),u_1-u_2 \big)_{0,\Omg} = - \big((\lambda_1-\lambda_2) (u_1-f),u_1-u_2\big)_{0,\Omg}.
\end{equation*}
Thanks to the ellipticity constant $\mu$ and since $\lambda_2 \in \mathcal V$, we get the bound
\begin{equation*}
 \|u_1-u_2\|^2_V  \leq C \|\lambda_1-\lambda_2\|_\infty \|u_1-f\|_{0,\Omg} \|u_1-u_2\|_{0,\Omg}.
\end{equation*}
Considering in addition the equivalence between the energy norm and the $L^2$-norm and estimate
\eqref{eq:stability-lower}, we get a bound on $\|u_1-f\|_{0,\Omg}$ and, consequently,
\begin{equation*}
 \|u_1-u_2\|_V  \leq C_f \|\lambda_1-\lambda_2\|_\infty \|\lambda_1\|_\infty,
\end{equation*}
which implies the local Lipschitz continuity of the solution mapping.}

Let now $h \in \mathcal U$ {be a given direction}, and $u_t$ and $u$ the unique solutions to \eqref{eq:rest_denoising} corresponding to $\lambda+th$ and $\lambda$, respectively. For $\epsilon$ and $t$ small enough, equation \eqref{eq:rest_denoising} is well-posed. Throughout the proof, we let $C > 0$
denote a generic positive constant.

{By taking the difference of equations  \eqref{eq:rest_denoising} corresponding to $\lambda+th$ and $\lambda$}, we get
\[
\mu (u_t-u,\psi)_V + \big((\lambda+th)(u_t-f)-\lambda(u-f),\psi\big)_{0,\Omg}=0, \quad \forall \psi \in V_c
\]
or, equivalently, {expanding all terms},
\begin{multline}\label{eq:diff}
  \mu \int_\Omgf \int_\Omgf \big( (u_t-u)(\xb)-(u_t-u)(\yb) \big)\big( \psi(\xb)-\psi(\yb) \big)\gamma(\xb,\yb)\,d\xb \,d\yb\\
+ \int_\Omg\lambda(\xb) (u_t-u)(\xb)\psi(\xb)\,d\xb 
+ t\int_\Omg h(\xb)u_t(\xb)\psi(\xb)\,d\xb\\ =t \int_\Omg h(\xb)f(\xb)\psi(\xb)\,d\xb,  \quad \forall \psi \in V_c.
\end{multline}
{By choosing the test function $\psi=u_t-u$ and since $\lambda \in \mathcal V$, we obtain the estimate}
\begin{align*}
\|u_t-u\|^2_V
&{} \leq C \left| t \int_\Omg h(x)(f(x)-u_t(x))(u_t-u)(x)\,d\xb \right|	\\
&{} \leq C t\| h\|_{\infty}\|f-u_t\|_{0,\Omg}\|u_t-u\|_{0,\Omg},
\end{align*}
which implies, {using the equivalence of norms for integrable localized kernels and the \emph{a-priori} bound for the lower-level solution \eqref{eq:stability-lower}, that}
\begin{align*}
\|u_t-u\|_{0,\Omg} \leq C t\| h\|_{\infty} \big\{\|f\|_{0,\Omg}+\|u_t\|_{0,\Omg} \big\} \leq C t\| h\|_{\infty} \big\{\|f\|_{0,\Omg}+K(\|\lambda\|_{\infty}+t\|h\|_{\infty}) \big\}.
\end{align*}
Therefore, the sequence $\{z_t\}_{t>0} = \{\frac{u_t-u}{t}\}_{t>0}$ is bounded in $V_C$, {for $t$ sufficiently small}, and there exists a subsequence (still denoted by $\{z_t\}$) such that $z_t\rightharpoonup z$ weakly in $V_C$. From \eqref{eq:diff} we {get, subtracting the term $\int\limits_\Omega h u \psi$ on both sides,}
\begin{align*}
  \int\limits_\Omgf \int\limits_\Omgf \mu \, \mcG \left(\frac{u_t-u}{t}\right) \mcG \psi
+ \int\limits_\Omg\lambda\left(\frac{u_t-u}{t}\right) \psi + \int\limits_\Omg h(u_t-u) \psi ={-\int\limits_\Omg} h(u-f) \psi, ~\forall \psi \in V_c,
\end{align*}
which implies that
\begin{align*}
\int_\Omgf \int_\Omgf \mu \, \mcG z_t\, \mcG \psi+ \int_\Omg\lambda z_t \psi
+ \int_\Omg h(u_t-u) \psi =-\int_\Omg h(u-f) \psi,\quad \forall \psi \in V_c.
\end{align*}
Taking the limit as $t \rightarrow 0,$ we obtain
\begin{equation} \label{eq: diff proof linearized eq}
\int_\Omgf \int_\Omgf \mu \, \mcG z\, \mcG \psi+ \int_\Omg\lambda\,z\,\psi 
= {-\int_\Omg} h(u-f)\, \psi, \quad \forall \psi \in V_c,
\end{equation}
{which has a unique solution $z \in V_c$ thanks to the ellipticity of the energy bilinear term.}
{By subtracting equation \eqref{eq: diff proof linearized eq} from equation \eqref{eq:diff} we get}
\begin{equation*}
    \mu \left( \frac{u_t-u}{t}-z, \psi \right)_V + \int_\Omega \lambda \left( \frac{u_t-u}{t}-z \right) \psi
= - \int_\Omega h (u_t-u) \psi,\quad \forall \psi \in V_c.
\end{equation*}
Finally, by choosing $\psi= \frac{u_t-u}{t}-z$ we obtain the estimate
\begin{equation*}
\left\| \frac{u_t-u}{t}-z \right\|_V \leq C \|h\|_{L^\infty} \|u_t-u\|_{0,\Omg}.
\end{equation*}
The continuity of the solution operator implies that $\left\| \frac{u_t-u}{t}-z \right\|_V \to 0$ as $t \to 0$, which concludes the proof.
\end{proof}

\subsection{Optimality system}
Thanks to the G\^ateaux differentiability of the solution operator we are able to derive an optimality system for the characterization of local optimal solutions of \eqref{eq:bilevel_reg_prob}.
\begin{theorem}
Let $(u,\lambda) \in V_c \times \mathcal C$ be an optimal solution to problem \eqref{eq:bilevel_reg_prob}. There exists an adjoint state \( p \in L^2_c(\Omgf)\) and Lagrange multipliers $\mu_\Omega^+, \mu_\Omega^- \in L^2(\Omega)$ and $\mu_\Gamma^+, \mu_\Gamma^- \in H^{1/2}(\Gamma)$ such that the following optimality system is satisfied:
\begin{subequations}\label{eq:sys-lambda-opt}
\begin{equation}
\mu (u,\psi)_V + \big(\lambda(u-f),\psi\big)_{0,\Omg} =0, \forall \psi \in V_C,
\end{equation}
\begin{equation}
	\mu (p,\phi)_V + (\lambda\,p,\phi)_{0,\Omg} =- (\nabla \ell(u),\phi)_{0,\Omg}, 	\forall \phi \in V_C,
\end{equation}
\begin{equation}
\begin{aligned}
-\beta \Delta \lambda + \beta \lambda {+(u-f)p} &= \mu_\Omega^+ -\mu_\Omega^- 
 \quad \text{in }\Omega,\\
 \beta \frac{\partial \lambda}{\partial \vec n} &= \mu_\Gamma^+ - \mu_\Gamma^-
 \quad \text{on }\Gamma:= \partial \Omega,
\end{aligned}
\end{equation}
\begin{equation}
\begin{aligned}
&  0 \leq \mu_\Omega^+(\xb) \perp \lambda(\xb) \geq 0, 
&& 0 \leq \mu_\Omega^-(\xb) \perp (\Lambda-\lambda(\xb)) \geq 0, 
     \quad \forall \xb \in \Omega,\\
&  0 \leq \mu_\Gamma^+(\xb) \perp \lambda(\xb) \geq 0, 
&& 0 \leq \mu_\Gamma^-(\xb) \perp (\Lambda-\lambda(\xb)) \geq 0, 
     \quad \forall \xb \in \Gamma.
\end{aligned}
\end{equation}
\end{subequations}
\end{theorem}
\begin{proof}
Let us consider the reduced cost functional
\begin{equation}\label{red-funct-parameter}
	j(\lambda) := \ell(u(\lambda)) + \frac\beta2 \|\lambda\|_{H^1}^2,
\end{equation}
where $u(\lambda)$ is the unique solution to the state equation \eqref{eq:rest_denoising} corresponding to $\lambda$. Taking the derivative of the reduced cost with respect to $\lambda$, we get
\begin{equation}
	j'(\lambda)h= \big( \nabla \ell(u(\lambda)), u'(\lambda)h\big)_{0,\Omg}+ \beta (\lambda,h)_{H^1}, \quad \forall h \in \mathcal U,
\end{equation}
{where, according to Theorem \ref{thm: diff solution operator lambda}, $u'(\lambda)h$ is the unique solution of the linearized equation}
\begin{equation}
    \mu \big(u'(\lambda)h,\psi\big)_V + \big(\lambda u'(\lambda)h,\psi\big)_{0,\Omg}
= - \big(h(u-f),\psi\big)_{0,\Omg}, \quad \forall \psi\in V_C.
\end{equation}
Using the adjoint equation
\begin{equation}
\mu (p,\phi)_V + (\lambda p,\phi)_{0,\Omg}=-(\nabla \ell(u),\phi)_{0,\Omg}, \quad \forall \phi\in V_c,
\end{equation}
which is uniquely solvable by the same arguments as in Theorem \ref{thm: diff solution operator lambda}, and choosing $\psi = u'(\lambda)h$, we obtain that
\begin{equation}
j'(\lambda)h= - \mu \big( u'(\lambda)h,p \big)_V- \big( \lambda u'(\lambda)h,p\big)_{0,\Omg}+ \beta (\lambda,h)_{H^1}.
\end{equation}
By using the linearized equation we then obtain
\begin{equation}
j'(\lambda)h= \int_\Omega (u-f)p h ~dx + \beta (\lambda,h)_{H^1}.
\end{equation}
The box constraints on the parameter $\lambda$ imply that the a first order necessary optimality condition is given by the following variational inequality:
\begin{equation}\label{ec:sys-lambda-opt-vi}
	j'(\lambda)(h-\lambda)= \int_\Omega (u-f)p (h-\lambda) ~dx + \beta (\lambda,h-\lambda)_{H^1} \geq 0, \quad \forall h \in \mathcal C.
\end{equation}
The latter corresponds to an obstacle problem with bilateral bounds. Integration by parts then yields
\begin{align*}
  (\lambda,v)_{H^1}&= (\lambda,v)_{0,\Omg} + (\nabla \lambda,\nabla v)_{0,\Omg}\\
  &= (\lambda,v)_{0,\Omg} + \int_\Gamma \frac{\partial \lambda}{\partial \vec n} v ~d \Gamma - (\Delta \lambda,v)_{0,\Omg} \;\;\forall\, v\in H^1(\Omg),
\end{align*}
where the extra regularity $\lambda \in H^2(\Omega)$ follows from \cite[Thm.~5.2]{troianiello2013elliptic}. Consequently, the variational inequality \eqref{ec:sys-lambda-opt-vi} can be written in strong form as
\begin{align*}
  -\beta \Delta \lambda + \beta \lambda {+(u-f)p} &= \mu_\Omega \quad \text{in }\Omega,\\
  \beta \frac{\partial \lambda}{\partial \vec n} &= \mu_\Gamma \quad \text{on }\Gamma,
\end{align*}
where the multipliers $\mu_\Omega \in L^2(\Omega)$ and $\mu_\Gamma \in H^{1/2}(\Gamma)$ satisfy
\begin{equation*}
(\mu_\Omega, v- \lambda) \geq 0, \; \forall v \in \mathcal C, \quad \text{ and } \quad (\mu_\Gamma, v- \lambda) \geq 0,  \; \forall v \in \mathcal C,
\end{equation*}
or, equivalently,
\begin{align*}
& \big(\mu_\Omega(\xb), v- \lambda(\xb)\big) \geq 0, 
  \quad \forall v \in [0, \Lambda] \subset \mathbb R, \text{ a.e. in }\Omega\\
& \big(\mu_\Gamma(x), v- \lambda(x)\big) \geq 0, 
  \quad \forall v \in [0,\Lambda] \subset \mathbb R, \text{ a.e. in }\Gamma.
\end{align*}
By decomposing $\mu_\Omega$ and $\mu_\Gamma$ in its positive and negative parts, we have
\begin{align*}
\mu_\Omega & = \mu_\Omega^+ - \mu_\Omega^-,\\
\mu_\Omega^+ & \geq 0, & \lambda(\xb) &\geq 0, &
\mu_\Omega^+(x) \lambda(x) &=0,\\
\mu_\Omega^- &\geq 0, & \lambda(x) &\leq \Lambda, &
\mu_\Omega^-(x) \big(\Lambda-\lambda(x) \big) &=0,
\end{align*}
and similarly for $\mu_\Gamma$.
\end{proof}

\subsection{The scalar parameter case}
When $\lambda\in\mbR^+ \cup \{0\}$, the Tikhonov regularization is no longer required  and the bilevel problem is given by
\begin{subequations}\label{eq:opt_prob_ESC}
\begin{equation}
\min_{0 \leq \lambda \leq b} \; \ell(u)
\end{equation}
\begin{equation}
{\rm s.t. } \mu (u,\psi)_V + \lambda(u-f,\psi)_{0,\Omg}=0, 
\quad \forall \;\psi \in V_c.
\end{equation}
\end{subequations}
Let the Lagrangian and its derivative with respect to $u$ be given by
\begin{equation*}
\mathbb L (u, \lambda,p) := \ell(u) + \mu (u,p)_V + \lambda(u-f,p)_{0,\Omg}
\end{equation*}
and
\begin{equation*}
\mathbb L_u(v) = \big(\nabla \ell(u),v\big)_{0,\Omg} + \mu (p,v)_V + \lambda(p,v)_{0,\Omg}.
\end{equation*}
It follows that
\begin{equation*}
	\mu (p,v)_V + \lambda(p,v)_{0,\Omg}=-\big(\nabla \ell(u),v\big)_{0,\Omg}, \quad \forall\; v \in V_c.
\end{equation*}
On the other hand, the derivative of $\mathbb L$ with respect to $\lambda$ is given by
\begin{equation*}
	\mathbb L_\lambda(h-\lambda) = (u-f,p) (h-\lambda)= (h- \lambda) \int_\Omega (u-f) p ~dx \geq 0, \quad \forall \;h \in [0,\Lambda].
\end{equation*}
{Thus, using Hilbert projection theorem, the optimality system may be written as follows:}
\begin{subequations}\label{opt_system:scalar_lambda}
\begin{align}\label{opt_system:scalar_lambda_a}
\mu (u,\psi)_V + \lambda(u-f,\psi)_{0,\Omg} &=0, & 
\forall \psi &\in V_c,\\
\mu (p,v)_V + \lambda(p,v)_{0,\Omg} &=-\big(\nabla \ell(u),v\big)_{0,\Omg}, &
\forall v &\in V_c,	\label{opt_system:scalar_lambda_b}\\
P_{[0,\Lambda]} \left(\lambda -c \int_\Omega (u-f) p ~dx\right) &= \lambda, & 
\forall c&>0, \label{opt_system:scalar_lambda_c}
\end{align}
\end{subequations}
where $P_{[0,\Lambda]}$ is the standard projection operator onto the interval $[0,\Lambda]$. {This reformulation turns out to be of numerical importance, since it enables the use of efficient first- and second-order iterative optimization methods.}

\section{Optimization with respect to the weights}\label{sec:opt-wrt-w}
In this section we introduce and analyze the bilevel problem for the identification of the optimal weight in a nonlocal means kernel. We consider a modified nonlocal means kernel where we restrict the integral to a bounded region, i.e.,
\begin{equation}\label{eq:w-kernel}
\gamma_w(\xb,\yb) = \exp\Big\{-
\int_{B_\rho(\zerob)} w(\taub) \big( f(\xb+\taub)-f(\yb+\taub) \big)^2 d\taub \Big\},
\end{equation}
where the $\ell^\infty$ ball $B_\rho$ is the patch of the image explored within the integration and $w \in \Uad:=\big\{v \in Y: 0 \leq w({\bf t}) \leq W, \text{ a.e. in }  B_\rho(\zerob) \big\}$, with $Y \subseteq L^2(B_\rho(\zerob))$ a closed subspace. This type of $\ell^\infty$ (or square) patches have been also used in \cite{Salmon2010}. Note that, to simplify the notation in the analysis, we embedded the parameter $\delta$ in the weight $w$; in Section \ref{sec:numerical-tests}, for each numerical test, we provide more details on the choice of kernel parameters.

Although our analysis is focused on a specific kernel, it can be extended to any exponential-type kernel and, in general, to kernels whose energy space is equivalent to $L^2$.

\subsection{Lower-level problem}\label{sec:lower-level-w}

For a given $\lambda > 0$ and $w \in\Uad$ we consider the following denoising problem
\begin{equation}\label{eq:lower-level-w}
\min_{u\in V_c^w} J(u,\lambda)= \frac{\mu}{2} \|u\|_{V^w}^2 + \frac{\lambda}{2} \int_\Omg (u-f)^2 \,d\xb,
\end{equation}
where the weight-dependent energy space is defined as \(V_c^w = \big\{ v \in L^2(\Omega \cup \Omega_I): \|v\|_{V^w} < \infty \big\}\) with
\[
\|v\|^2_{V^w}  := \displaystyle \int_\Omgf \int_\Omgf	\big(v(\xb)-v(\yb)\big)^2
\gamma_w(\xb,\yb) \,d\yb d\xb.
\]
Note that the spaces $V_c^w$ are also equivalent to $L^2_c$ for any $w$. Moreover, the equivalence constants are independent of $w$, thanks to the bilateral box constraints. By proceeding in a similar manner as in Theorem \ref{thm:lower-level-wellp}, it can be readily verified that, for every \(w \in \Uad \), there exists a unique solution \(u \in V^w\) for the lower-level problem \eqref{eq:lower-level}. Moreover, the strict convexity and differentiability of the fidelity term yields the following necessary and sufficient optimality condition
\begin{equation}\label{eq:lower-optimality-w}
	\mu (u,\psi)_{V^w} + \big(\lambda(u-f),\psi\big)_{0,\Omg}=0, 		\quad \forall \psi \in V_c^w,
\end{equation}
where
\begin{equation}
	(u,\psi)_{V^w} = \displaystyle \int_\Omgf \int_\Omgf 	\big(v(\xb)-v(\yb)\big) \big(\psi(\xb)-\psi(\yb)\big) \gamma_w(\xb,\yb)\,d\yb d\xb,
\end{equation}
and the following a-priori bound
\begin{equation} \label{eq: apriori weight}
	\|u\|_{V^w} \leq C_\lambda \|f\|_{0,\Omg}.
\end{equation}

\subsection{Bilevel problem}\label{sec:bilevel-w}

We consider the following bilevel optimization problem for the estimation of the optimal kernel weight in \eqref{eq:w-kernel}.
\begin{equation}\label{eq:weight-opt}
	\min_{(w,u)\in\mathcal T_{ad}} \ell (u),
\end{equation}
where the feasible set is given by $\mathcal T_{ad}:= \big\{ (u,w): w\in \Uad \text{ and }\eqref{eq:lower-optimality-w} \text{ holds} \big\}$. 
\begin{theorem}\label{thm:bi-level-existence-w}
	Let $Y \subset L^2(B_\rho(\zerob))$ be a finite-dimensional subspace. The bilevel problem \eqref{eq:weight-opt} admits a solution \((\ua,\wa) \in \mathcal T_{ad}\).
\end{theorem}
\begin{proof}
Since the functional is bounded from below, the box constraints and the a-priori estimate \eqref{eq: apriori weight} imply that there exists a minimizing sequence $\{(w_n, u_n)\} \in \mathcal T_{ad}$ that is uniformly bounded. Moreover, the box constraints and the equivalence of spaces imply that the sequence $\{u_n\}$ is also bounded in $L^2(\Omgf)$. Thus, there exists a subsequence, that we still denote by $\{(w_n, u_n)\}$, {and a limit point $(\wa,\ua) \in Y \times L^2(\Omgf)$ such that
$w_n \weak \wa$ weakly in $Y$ and $u_n \weak \ua$ weakly in $L^2(\Omgf)$. Since $Y$ is finite-dimensional, it follows that $w_n \to \wa$ strongly in $Y$}

We next show that $(\ua,\wa)\in \mathcal T_{ad}$. Since $\Uad$ is weakly closed, $\wa$ also satisfies the box constraints. Moreover, since $f \in L^\infty(\Omgf)$, it follows that
\begin{equation*}
	\int\limits_{B_\rho(\zerob)} -w_n (\taub) \big( f(\xb+\taub)-f(\yb+\taub) \big)^2 d\taub \to \int\limits_{B_\rho(\zerob)} -\wa(\taub) \big( f(\xb+\taub)-f(\yb+\taub) \big)^2 d\taub.
\end{equation*}
which implies that $\gamma_{w_n}(\xb,\yb) \to \gamma_{\wa}(\xb,\yb), \, \forall \xb,\yb$. {Moreover, it can be verified the limit holds uniformly and, therefore, $\gamma_{w_n} \to \gamma_{\wa}$ in $L^\infty(\Omgf \times \Omgf)$. Indeed,
\begin{align*}
|\gamma_{w_n}(\xb,\yb) - \gamma_{\wa}(\xb,\yb)| &= \left| \int\limits_{B_\rho(\zerob)} -(w_n (\taub)-\wa(\taub)) \big( f(\xb+\taub)-f(\yb+\taub) \big)^2 d\taub \right|\\
& \leq \int\limits_{B_\rho(\zerob)} |w_n (\taub)-\wa(\taub)| \big( f(\xb+\taub)-f(\yb+\taub) \big)^2 d\taub\\
& \leq  4 \|f\|_\infty^2 \rho \|w_n - \wa\|_{L^2},
\end{align*}
i.e., the bound is independent of $(\xb,\yb)$.
}

Let us recall that each pair $(u_n,w_n)$ solves
\begin{equation}
	\mu\int\limits_\Omgf\int\limits_\Omgf (u_n-u_n')(v-v') \gamma_{w_n}(\xb,\yb) d\yb\,d\xb
+\lambda\int\limits_\Omg u_n\,v\,d\xb = \lambda\int\limits_\Omg f\,v\,d\xb, ~\forall v \in V_c^{w_n}
\end{equation}
where, to simplify the notation, we used $v:=v(\xb)$ and $v':=v(\yb)$. {Thanks to the weak convergence of $u_n \weak \ua$ in $L^2(\Omgf)$ and the strong convergence of $\gamma_{w_n}(\xb,\yb) \to \gamma_{\wa}(\xb,\yb)$ in $L^\infty(\Omgf \times \Omgf)$, we may pass to the limit, as $n\to\infty$, in the previous equation and get that $(\ua,\wa)\in \mathcal T_{ad}$, i.e.,}
%
\begin{equation}
	\mu\int\limits_\Omgf\int\limits_\Omgf (\ua-\ua')(v-v')\gamma_{\wa}(\xb,\yb) d\yb\,d\xb
	+\lambda\int\limits_\Omg (\ua-f)v\,d\xb = 0, ~\forall v \in V_c^{w^*}.
\end{equation}

Finally, since the loss function is convex and continuous, it is weakly lower semicontinuous, and, thus, $(\ua,\wa)$ is a solution of \eqref{eq:weight-opt}.
\end{proof}

\subsubsection{Differentiability of the solution operator}

We first prove a lemma that will be useful in the proof of differentiability.
\begin{lemma}\label{thm:ut-norm-bound}
Let $w \in \Uad$ and $h \in Y$ be a feasible direction, i.e., there exists some  $t\in\mbR^+$ such that $w+th \in \Uad$. Then, the weak solution of problem
\begin{equation}\label{eq:ut-problem}
	\mcL_t u_t + \lambda(u_t-f)=0
\end{equation}
with
\begin{equation}
	\mcL_t v(\xb) = 2\mu\int\limits_{\Omgf\cap B_\veps(\xb)} (v'-v) \gamma_{w+th}(\xb,\yb)d\yb,
\end{equation}
satisfies the estimate
\begin{equation}\label{eq:ut-norm}
	\mu\|u_t\|_{V^w}^2+ \lambda \|u_t\|_{L^2(\Omgf)}^2\leq \lambda\|f\|_{L^2(\Omega)} \|u_t\|_{L^2(\Omgf)}.
\end{equation}
\end{lemma}
\begin{proof}
The weak formulation of \eqref{eq:ut-problem} reads
\begin{equation}\label{eq:weak-t-form}
\begin{aligned}
\mu\int\limits_\Omgf\int\limits_\Omgf (u_t-u_t')(v-v') &\gamma_{w+th}(\xb,\yb) d\yb\,d\xb
+\lambda\int\limits_\Omg (u_t-f)v\,d\xb = 0, \quad \forall \, v\in V_t,
	\end{aligned}
\end{equation}
where $V_t$ is the energy space induced by using the weight $(w+th)$. For $v=u_t$, the result follows from
\begin{align*}
\mu\int\limits_\Omgf\int\limits_\Omgf  &(u_t-u_t')^2 \gamma_{w+th}(\xb,\yb) d\yb\,d\xb+\lambda \|u_t\|^2_{0,\Omgf} \\[2mm]
&=  \mu\|u_t\|_{V_t}^2+\lambda \|u_t\|^2_{0,\Omgf}
\leq \lambda \|f\|_{L^2(\Omega)} \|u_t\|_{0,\Omgf}.
\end{align*}
\end{proof}
\begin{remark}
The last result implies that
\begin{equation*}
\mu \|u_t\|_{V^w}^2 \leq \lambda \|f\|_{L^2(\Omega)} \|u_t\|_{0,\Omgf} \quad \text{and} \quad \|u_t\|_{L^2(\Omgf)} \leq \|f\|_{L^2(\Omega)},
\end{equation*}
and, consequently, $\|u_t\|_{V^w} \leq C(\lambda) \|f\|_{L^2(\Omega)}.$
\end{remark}

Next, we prove that the sequence $\big\{\frac{u_t-u}{t}\big\}$ has a bounded $L^2$ norm and, thus, contains a weakly convergent subsequence.
\smallskip
\begin{lemma}\label{thm:zt-norm-bound}
Let $w \in \Uad$ and $h \in Y$ be a feasible direction. The sequence $\{z_t\}=\big\{ (u_t-u)/t \big\}$, where $u$, $u_t$ are the solutions to \eqref{eq:lower-optimality-w} and \eqref{eq:ut-problem} and {$t$ is sufficiently small}, is bounded in $L^2(\Omgf)$.
\end{lemma}
\begin{proof}
By subtracting the weak forms \eqref{eq:lower-optimality-w} and \eqref{eq:weak-t-form}, and using the equivalence of spaces, we obtain
\begin{equation}\label{eq:weak-t-forms-difference}
\begin{aligned}
	&\mu\int\limits_\Omgf\int\limits_\Omgf (u_t-u_t')(v-v')\gamma_{w+th}(\xb,\yb) d\yb\,d\xb \\[2mm]
	&\quad-\mu\int\limits_\Omgf\int\limits_\Omgf (u-u')(v-v')\gamma_{w}(\xb,\yb)d\yb\,d\xb
	+\;\lambda\int\limits_\Omg (u_t-u)v\,d\xb = 0, \;\; \forall \, v\in V^w.
\end{aligned}
\end{equation}
Thus,
\begin{align}
\nonumber
	&\mu\int\limits_\Omgf\int\limits_\Omgf (u_t-u-u_t'+u')(v-v')\gamma_{w}(\xb,\yb)d\yb\,d\xb\, +
 	\\[2mm]
 \nonumber
	&\mu\int\limits_\Omgf\int\limits_\Omgf (u_t-u_t')(v-v')
	\Big[\gamma_{w+th}(\xb,\yb) -\gamma_{w}(\xb,\yb)\Big] d\yb\,d\xb
	\\[2mm]
	&\qquad\qquad+ {\lambda}\int\limits_\Omg (u_t-u)v\,d\xb=0, \;\; \forall \, v\in V^w.		\label{eq:weak-t-forms-difference-2}
\end{align}
Choosing $v=u_t-u$ and dividing all expressions by $t$, we get
\begin{multline*}
	\frac{\mu}{t}\|u_t-u\|^2_{V^{w}}+ \frac{\lambda}{t}\|u_t-u\|^2_{0,\Omgf}\\
	+\mu\int\limits_\Omgf\int\limits_\Omgf (u_t-u_t')(u_t-u-u_t'+u') \frac{1}{t}\Big[\gamma_{w+th}(\xb,\yb)-\gamma_{w}(\xb,\yb)\Big]d\yb\,d\xb=0.
\end{multline*}

{By using the differentiability of the exponential function as a superposition operator (see, e.g., \cite[Section 4.3.2]{troltzsch2010optimal}), the bilateral box constraints in $\Uad$ and the equivalence of norms, we obtain}
\begin{multline*}
	\frac{C}{t} \|u_t-u \|^2_{0,\Omgf}  + \frac{\lambda}{t} \|u_t-u \|^2_{0,\Omgf}\\ \leq
	\tilde C \left( \|h\|_{L^2(B_\rho(\zerob))} \|f\|^2_{\infty} +\frac{o(t)}{t} \right) \|u_t\|_{0,\Omgf}\|u_t-u\|_{0,\Omgf},
\end{multline*}
which, combined with \eqref{eq:ut-norm}, implies that
\begin{equation*}
\left\| \frac{u_t-u}{t} \right\|_{0,\Omgf} \leq
C \|h\|_{L^2(B_\rho(\zerob))} \|f\|^3_{\infty} + \frac{o(t)}{t} \|f\|_{\infty}.
\end{equation*}
\end{proof}

The lemma above guarantees existence of a weakly convergent subsequence (denoted the same) and of a limit point $\za$ such that $z_t \weak \za$ in $L^2(\Omgf)$. {As a byproduct we also get that $u_t \to u$ in $L^2(\Omgf)$ as $t \to 0$.}

In the following proposition we derive the linearized equation for $\za$.
\begin{proposition}\label{thm:z-equation}
Let $\za$ be such that $z_t\weak\za$ in $L^2(\Omgf)$. Then, $\za$ corresponds to the unique solution of the linearized equation
\begin{equation}\label{eq:z-equation}
	\mu(\za,v)_V  + \mu(u,v)_{\widetilde V}+ \lambda(\za,v)_{0,\Omg}=0, \quad \forall v \in V^w,
\end{equation}
with $(u,v)_{\widetilde V}= \int\limits_\Omgf\int\limits_\Omgf (u-u')(v-v') \widetilde \gamma_h(\xb,\yb) (\xb,\yb) d\yb\,d\xb$, where $\widetilde \gamma_h(\xb,\yb)$ is the linearized kernel given by
\begin{equation}\label{eq:lin-kernel}
    \widetilde \gamma_h(\xb,\yb)=\gamma_{w}(\xb,\yb) ~\int_{B_\rho(\zerob)} -h(\taub) \big( f(\xb+ \taub)-f(\yb+\taub) \big)^2 ~d \taub.
\end{equation}
\end{proposition}
\begin{proof}
By \eqref{eq:weak-t-forms-difference} we have
\begin{align*}
	&\mu \int\limits_\Omgf\int\limits_\Omgf (u_t-u_t')(v-v')\gamma_{w+th}(\xb,\yb) d\yb\,d\xb \\[2mm]
	& \quad- \mu\int\limits_\Omgf\int\limits_\Omgf (u-u')(v-v')\gamma_{w}(\xb,\yb) d\yb\,d\xb
	+   \lambda\int\limits_\Omg (u_t-u)v\,d\xb=0.
\end{align*}
%
Adding and subtracting $(u,v)_{V_t}$ and dividing both sides by $t$, we get
\begin{align*}
	& \mu\int\limits_\Omgf\int\limits_\Omgf \frac{1}{t} \big( (u_t-u) - (u_t'-u') \big)(v-v')\gamma_{w+th}(\xb,\yb)d\yb\,d\xb\\[2mm]
	+\, &\mu\int\limits_\Omgf\int\limits_\Omgf (u-u')(v-v')
    \frac{1}{t} \big[\gamma_{w+th}(\xb,\yb) - \gamma_{w}(\xb,\yb) \big]d\yb\,d\xb\\
    &\qquad +\frac{\lambda}{t}(u_t-u,v)_{0,\Omg} =0,
\end{align*}
The weak convergence $\frac{u_t-u}{t} \weak z^*$ in $L^2(\Omgf)$, the strong convergence $w+th \to w$ and the continuity and differentiability of the exponential superposition operator imply that the limit as $t\to 0$ of the previous equation is given by
\begin{align*}
\mu(\za,v)_V + \lambda(\za,v)_{0,\Omg}
\quad + \mu\int\limits_\Omgf\int\limits_\Omgf (u-u')(v-v')\widetilde \gamma_h(\xb,\yb) d\yb\,d\xb=0.
\end{align*}
Uniqueness follows as for the state equation thanks to the ellipticity of the energy terms.
\end{proof}

The following theorem finalizes the differentiability result.
\begin{theorem}
Let $\func{S_w}{\Uad}{V^w}$ be the solution operator which maps $w$ into the corresponding solution \(u \in V^w\) to equation \eqref{eq:lower-optimality-w}. Then the operator $S_w$ is directionally differentiable in any feasible direction {and the directional derivative $z^*=S_w' h$ at $w$ in direction $h \in Y$ is given by the unique solution of the linearized equation \eqref{eq:z-equation}.}
\end{theorem}
\begin{proof}
Thanks to lemmas \ref{thm:ut-norm-bound}, \ref{thm:zt-norm-bound} and Proposition \ref{thm:z-equation}, it only remains to prove that
\begin{equation*}
\left\| \frac{u_t-u}{t}-z^*\right\|_{L^2(\Omgf)} \to 0 \quad \text{ as }t \to 0.
\end{equation*}
From equations  \eqref{eq:weak-t-forms-difference-2}  and \eqref{eq:z-equation}  we obtain that the difference $\zeta:= \frac{u_t-u}{t}-z^*$ is solution of the equation
\begin{multline*}
	\mu(\zeta,v)_{V^w} + \frac{\mu}{t} (u_t,v)_{V_t} -\frac{\mu}{t} (u_t,v)_{V^w} - \mu(u_t,v)_{\tilde V}
	+ \mu(u_t,v)_{\tilde V} - \mu(u,v)_{\tilde V} + \lambda \int_\Omega \zeta v =0,
\end{multline*}
or, equivalently,
\begin{multline*}
	\mu(\zeta,v)_{V^w} + \mu\int\limits_\Omgf\int\limits_\Omgf (u_t-u_t')(v-v')
	\left[ \frac{1}{t} \gamma_{w+th}(\xb,\yb)		\right.
      \left. -\frac{1}{t} \gamma_{w}(\xb,\yb)  -\widetilde \gamma_h(\xb,\yb) \right] d\yb\,d\xb
      \\
      +\mu(u_t-u,v)_{\tilde V} + \lambda \int_\Omega \zeta v =0.
\end{multline*}
By choosing $v= \zeta$ we have
\begin{multline*}
  \mu\|\zeta\|^2_{V^w} + \lambda \int_\Omega \zeta^2 = -\mu(u_t-u,\zeta)_{\tilde V} \\
  -\mu\int\limits_\Omgf\int\limits_\Omgf (u_t-u_t')(\zeta-\zeta')
      \left[ \frac{1}{t} \gamma_{w+th}(\xb,\yb)	\right.
      \left. -\frac{1}{t} \gamma_{w}(\xb,\yb) -\widetilde \gamma_h(\xb,\yb) \right] d\yb\,d\xb.
\end{multline*}
{The equivalence of norms}, the convergence $u_t \to u$ in $L^2(\Omgf)$ and the differentiability of the exponential Nemitski superposition operator, allow us to take the limit as $t \to \infty$, which yields the result.
\end{proof}

\subsection{Optimality system}
The differentiability of the solution operator allows us to derive an optimality system that characterizes local  optimal solutions of \eqref{eq:weight-opt}.

\begin{theorem}
Let \((u,w)\) be an optimal solution to problem \eqref{eq:weight-opt}. Then, there exists a Lagrange multiplier \(p \in L^2_c (\Omgf)\) such that the following optimality system is satisfied.
\begin{equation}\label{eq:sys-weights-opt}
\begin{array}{llll}
\displaystyle\mu(u,\psi)_{V^w} + \big(\lambda(u-f),\psi\big)_{0,\Omg}=0,
   & \forall \psi \in V^w, \\[3mm]
\mu(p,\phi)_{V^w} + (\lambda\,p,\phi)_{0,\Omg} + \big(\nabla \ell(u),\phi\big)_{0,\Omg} =0,  	
   & \forall \phi \in V^w, \\[3mm]
\displaystyle\mu \int\limits_\Omgf\int\limits_\Omgf \Big[ \big(u(w)-u(w)'\big)(p-p')
\widetilde \gamma_{(h-w)}(\xb,\yb) \Big]d\yb\,d\xb \geq 0,
   & \forall h \in \Uad.
\end{array}
\end{equation}
\end{theorem}
\begin{proof}
Let the reduced cost functional be defined as
\begin{equation}\label{red-funct-weight}
j(w) := \ell\big(u(w)\big),
\end{equation}
where \(u(w)\) is the unique solution to the state equation \eqref{eq:lower-optimality-w} corresponding to \(w\). By taking the derivative of the reduced cost with respect to \(w\), in direction $h$, we get
\begin{equation*}
j'(w)h= \big( \nabla \ell (u(w)), u'(w)h\big)_{0,\Omg},
\end{equation*}
where \(u'(w)h\) is the unique solution to the linearized equation
\begin{equation*}
\mu\big(u'(w)h,\psi\big)_{V^w} + \mu\big(u(w),\psi\big)_{\tilde V}+\lambda \big(u'(w)h,\psi\big)_{0,\Omg}=0, \quad \forall \psi\in V^w.
\end{equation*}
As in \eqref{opt_system:scalar_lambda_b}, the adjoint equation is given by
\begin{equation}
\mu(p,v)_{V^w} + \lambda(p,v)_{0,\Omg}=-\big(\nabla \ell(u),v\big)_{0,\Omg}, 
\quad \forall v \in V^w,
\end{equation}
which is uniquely solvable by the same arguments as in Theorem \ref{thm:lower-level-wellp}. Thus we obtain
\begin{equation*}
j'(w)h= - \mu\big( u'(w)h,p \big)_{V^w} - \lambda \big(  u'(w)h,p\big)_{0,\Omg}.
\end{equation*}
Using the linearized equation and considering the box constraints on \(w\), we then get the first order necessary optimality condition
\begin{align}\nonumber
j'(w)(h-w) &=
\mu \int\limits_\Omgf\int\limits_\Omgf \Big[ \big(u(w)-u(w)'\big)(p-p')
\widetilde \gamma_{(h-w)}(\xb,\yb)  \Big]	d\yb\,d\xb \geq 0,
\end{align}
for all $h \in \Uad$.
\end{proof}
\begin{remark}
When $w$ is a scalar parameter, the last expression in the optimality system may be replaced with
\begin{align*}
P_{[0,W]} \Big( w - c \big(u,p\big)_{\hat V}  \Big) = w, \quad \forall c > 0,
\end{align*}
where $P_{[0,W]}$ is the standard projection operator onto the interval $[0,W]$ and
\begin{multline*}
(u,p)_{\hat V}:=\\ \mu \int\limits_\Omgf\int\limits_\Omgf \big(u-u'\big)(p-p') \gamma_{w}(\xb,\yb)
\left[ \int\limits_{B_\rho(\zerob)} - \big( f(\xb+ \taub)-f(\yb+\taub) \big)^2 ~d \taub \right]	d\yb\,d\xb.
\end{multline*}
This enables the use of projection algorithms for solving the bilevel problem.
\end{remark}

\section{Numerical tests}\label{sec:numerical-tests}
In this section we propose a numerical algorithm for the solution of the bilevel problem and illustrate our approach with several numerical tests. First, we describe the discretization of the optimality system and the technique used to evaluate the kernel. Then, we present the optimization algorithm which relies on a trust-region scheme with active set prediction and limited memory BFGS matrices. The section concludes with results of numerical computations illustrating the main features of the proposed approach.

\subsection{Discretization}\label{sec:FEM}
As we are interested in developing a second-order algorithm to solve the optimality system \eqref{opt_system:scalar_lambda}, we consider an \(H^1\)--Riesz representation of the derivative, $j'(\lambda)h$ of the reduced cost, where
\begin{equation}
j'(\lambda)h= \int_\Omega (u-f)p h ~dx + \beta (\lambda,h)_{H^1}.
\end{equation}
The Riesz representative is then given by the solution of the following equation
\begin{equation}
	(y,h)_{0,\Omg} + (\nabla y, \nabla h)_{0,\Omg} = \big( (u-f)p, h \big)_{0,\Omg} + \beta (\lambda,h)_{H^1}.
\end{equation}
%
Since we are interested in possibly discontinuous nonlocal solutions we use two finite element bases to approximate \(u\) and $p$ in \(L^2 (\Omg)\) and \(y,\lambda\) in \(H^1 (\Omg)\).  Specifically, we consider piecewise constant and piecewise {\color{black}bilinear elements}, respectively, defined over a partition of \(\Omgf\), which we denote as \(\mathcal{T}^h\).
Throughout this section we denote discretized quantities by using the superscript \(h\), and we fix \(\mu = \frac{1}{2}\).

%
{\color{black}
Let \(\mathcal{T}^h\) be a partition of \(\Omgf\) into regular non-overlapping {\color{black}rectangles}.
We consider the spaces
\begin{align}
	V^h_c 	&:= \big\{ \varphi \in L^2(\Omgf):\, \varphi\big|_{T} \in {\color{black}\mathbb P_0} \; \text{and} \; \varphi\big|_{\Omg_I} = 0, \, \forall\, T \in \mathcal{T}^h  \big\}, 	\\[0.5em]
	Y^h		&:= \big\{ \phi \in \mathcal{C}(\overline{\Omg}):\, \phi\big|_{T} \in {\color{black}\mathbb Q_1}, \, \forall\, T \cap \oOmg \in \mathcal{T}^h  \big\}.
\end{align}
{\color{black}where $\mathbb P_0$ indicates piecewise constant polynomials and $\mathbb Q_1$ piecewise bilinear ones.}

Note that \(V_c^h\) is a subspace of step functions and naturally \(V^h_c \subset V_c\). Without loss of generality, we consider unit volume elements. Now, for every rectangle \( T_i \in \mathcal{T}^h\), and letting \(u^h := \sum_{T_j \in \mathcal{T}^h} u_j^h \varphi_j\), we have that
\begin{subequations}
\begin{align*}
	\mu (u^h,\varphi_i)_{V_c^h} &=
	\int_{\Omega\cup\Omega_I} \int_{\Omega\cup\Omega_I} \big( u(\xb) - u(\yb)  \big) \gamma (\xb,\yb) \varphi_i(\xb) \, d\yb\,d\xb
	\\
	&= \int_{\Omega\cup\Omega_I} \int_{\Omega\cup\Omega_I}  \sum_{T_j \in \mathcal{T}^h} u_j^h \varphi_j (\xb) \varphi_i(\xb) \gamma (\xb,\yb)  \, d\yb\,d\xb
	\\
	&\qquad
	-\int_{\Omega\cup\Omega_I} \varphi_i(\xb)
	\int_{\Omega\cup\Omega_I} \sum_{T_j \in \mathcal{T}^h} u_j^h \varphi_j(\yb)   \gamma (\xb,\yb) \, d\yb\,d\xb.
\end{align*}
Since \( \varphi_i (\xb) \varphi_j(\xb) = 0 \), whenever \( i\neq j\), we get
\begin{align*}
	\mu (u^h,\varphi_i)_{V_c^h} &= \int_{T_i} \int_{\Omega}  u_i^h \varphi_j (\xb) \varphi_i(\xb) \gamma (\xb,\yb)  \, d\yb\,d\xb
	\\
	&\qquad
	-\int_{T_i} \varphi_i(\xb)
	\int_{\Omega\cup\Omega_I} \sum_{T_j \in \mathcal{T}^h}  u_j^h \varphi_j(\yb)   \gamma (\xb,\yb) \, d\yb\,d\xb
	\\
	&= u_i^h \int_{T_i} \int_{\Omega}  \gamma (\xb,\yb)  \, d\yb\,d\xb -\int_{T_i} \sum_{T_j \in \mathcal{T}^h}   u_j^h
	\int_{\Omega \cap T_j}  \gamma (\xb,\yb) \, d\yb\,d\xb.
\end{align*}
Now, let \(\gamma^h\) be a discrete approximation of the kernel in \(V_c^h\). Then
\begin{align*}
	\mu (u^h,\varphi_i)_{V_c^h}
	&= u_i^h \sum_{T_j \in \mathcal{T}^h}  \gamma^h_{i,j}   -  \sum_{T_j \in \mathcal{T}^h}   u_j^h  \gamma^h_{i,j}.
\end{align*}
Notice that \((u^h,\varphi_i)_{V_c^h} = 0\) whenever \( T_i \subset \Omega_I\). Here we recall that a constant factor coming from the discretization of \eqref{ec:non_local_product} is cancelled with \(\mu\).
\end{subequations}
}

\medskip
Choosing $w({\bf t})=\delta^{-2}$, the discrete analogue of the optimality system is given by
\begin{subequations}\label{sys:nonlocal-by-triangles}
\begin{align}
	\widetilde\lambda^h_i  u^h_i + \eta_i u^h_i - \sum_{T_j\in \mathcal{T}} u^h_j \gamma^h_{i,j}  &= \widetilde\lambda^h_i f^h_i,
	&&\forall T_i \in \mathcal{T}^h,\\
	\widetilde\lambda^h_i  p^h_i + \eta_i p_i^h - \sum_{T_j\in \mathcal{T}} p^h_j \gamma^h_{i,j} &= u_{T,i} - u^h_i,		&&\forall T_i \in \mathcal{T}^h;
\end{align}
\end{subequations}
where $u^h_i$, $p_i^h$,  \(\widetilde\lambda^h_i\), and $f_i^h$ are the values of the approximate nonlocal state, nonlocal adjoint, fidelity weight, and forcing term at triangle $T_i$, and
\(\gamma^h_{i,j}\) is the value of the approximate kernel for $\xb\in T_i$ and \(\yb\in T_j \), and \(\eta_i := \sum_{T_j} \gamma^h_{i,j} \). The evaluation of $f_i^h$, $\widetilde\lambda_i^h$ and $\gamma_{i,j}^h$ depends on the location of the pixels and is described in detail in the next section.

We rewrite system \eqref{sys:nonlocal-by-triangles} in a more compact form as follows
\begin{subequations}
\label{eq:discretized-os}
\begin{align}
	\big(\diag(\lambdab) + \diag(\etab) -  \Gamma\big) \ub &= \lambdab\circ\fb,
	\\
	\big(\diag(\lambdab) + \diag(\etab) -  \Gamma\big) \pb &= \ub_T - \ub,
\end{align}
\end{subequations}
where, for a vector $\vb$, \(\diag(\vb)\) is the \(n\times n\) diagonal matrix whose diagonal entries are the components of \(\vb\), and \(\circ\) is the Hadamard product, i.e., \(\vb\circ \wb := \begin{pmatrix} v_1 w_1 , \ldots , v_n w_n \end{pmatrix}^\top \), for any two vectors \( \vb,\wb \in \mathbb{R}^n \).
The bold notation refers to the vectors whose components are the values of the variables at the DOFs. The matrix $\Gamma$ is such that $\Gamma_{i,j}:=\gamma_{i,j}^h$.

The equation for the gradient of the reduced cost functional for problem \eqref{eq:bilevel_reg_prob} is discretized as
\begin{align*}
    (A+B) \yb         &= F(\ub, \lambdab), \quad\qquad\qquad \text{where} 
    \\[4mm]
     A_{i,j}          &= \int_\Omg \nabla \phi_i \nabla \phi_j d\xb, \qquad
    B_{i,j}            =  \int_\Omg  \phi_i \phi_j d\xb, \qquad \text{and}
    \\[4mm]
    F(\ub,\lambdab)_i &= \int_\Omega \big( (u^h-f^h)p^h + \beta \lambda^h \big)
              \phi_i d\xb + \sum_{j=1}^n \lambda^h_i \int_\Omega 
              \nabla \phi_i \cdot \nabla \phi_j  d\xb
\end{align*}
being \(\phi_i\) and \(\phi_j\) elements of the finite element basis associated with triangles \(T_i\) and \(T_j\), respectively, and $\lambda_i^h$ the values of the finite element solution $\lambda^h$ at the degrees of freedom.

Second, we consider the optimization with respect to $w$; we let \(\lambda\in \mathbb{R}^+\) and \(w\in\mbR^+ \cup \{0\} \). Thus, {\eqref{eq:sys-weights-opt} becomes}
\begin{subequations}
\label{eq:discretized-os-weights}
\begin{align}
\big(\lambda I_n + \diag(\etab_w) -  \Gamma_w\big) \ub_w &= \lambda \fb \\
\big(\lambda I_n + \diag(\etab_w) -  \Gamma_w\big) \pb_w &= \ub_T - \ub_w,
\end{align}
\end{subequations}
where we use the sub-index \(w\) to indicate the dependency of the kernel on \(w\). The gradient of the reduced cost functional for problem \eqref{eq:weight-opt} is discretized as
\begin{equation}\label{ec:disc-weight-derivative}
j'(w)^h= \pb_w \cdot \big(\diag( \widehat \etab) - \widehat \Gamma_w\big)\ub_w.
\end{equation}
for which \(\widehat\etab_i := \sum_{T_j} \widehat\Gamma_{w,i,j}\) and \( \widehat \Gamma_{w,i,j} = \widehat\gamma_{w,i,j}^h\) is a discretization of \( \widetilde \gamma_{(1)} \), defined in \eqref{eq:lin-kernel}.

\smallskip
Finally, we mention that in our numerical experiments we precondition the nonlocal systems by the following diagonal precondititoner $P$: let \(a_{i,j}\) be the entries of the matrix of the system for \(u^h\), we have:
$$
P_{i,i} := \left( \sum_{j} a_{i,j}^2 \right)^{-1/2}.
$$
Using this preconditioner, we solved the nonlocal systems with the Loose Generalized Minimal Residual Method (LGMRES).

\subsubsection{Evaluating the modified nonlocal means kernel}
For a given image \( f:\Omega \mapsto [0,255]\), where \(\Omg = [0,N]\times[0,M]\), the evaluation of the modified nonlocal means kernel requires the identification of several patches within the image. In this section, we describe how to define those regions and efficiently evaluate the kernel. Since the kernel decays exponentially away from the origin, we consider a relatively small radius $\rho$. Also, {\color{black} note that for both the optimization with respect to $\lambda$ and $w$ we only consider constant weights, i.e. $w({\bf t})=w\in\mbR^+$. This fact is particularly helpful for a computationally cheap evaluation of the corresponding matrix for different values of $w$; in fact, in this setting we have $\gamma_w(\xb,\yb)=\gamma_2(\xb,\yb)^w$, i.e. the weighted kernel is the $w$-th power of the nonlocal means kernel (as defined in \eqref{model: nonlocal means}) whose corresponding matrix can be computed once and for all.}

By definition, \(\Omgi = [-\veps,N+\veps]\times [-\veps,M+\veps] \, \setminus \, \Omega\); we extend \(f\) to zero outside \(\Omg\). 
{\color{black}We label each pixel by the integer $\ib$;} as illustrated in Figure \ref{fig:Pixels}, pixel $\ib$ is located at the upper-left corner of the element $T_i\in\mathcal T^h$ so that every element is associated with one pixel. 

In this setting, the approximation $\widetilde\lambda_i^h$ introduced above consists in the value of $\lambda^h$ at the pixel corresponding to element $T_i$, i.e. pixel $\ib$.
\begin{figure}[H]
 \centering
 \includegraphics[width=0.8\textwidth]{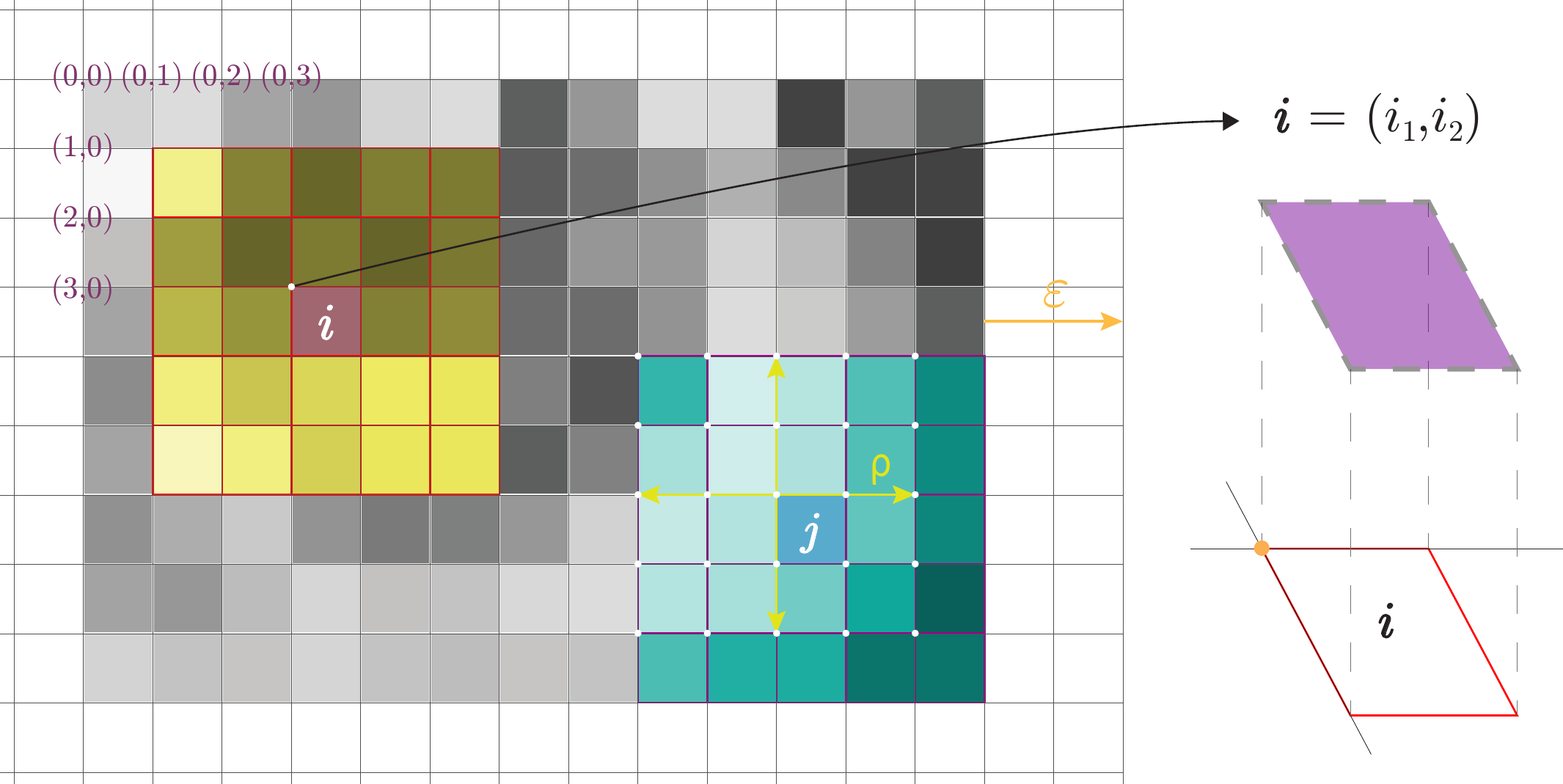}
 \\[-0.3em]
 \caption{Description of patches and pixel discretization in relation to the finite element grid. On the right, a pixel is depicted with its associated element, they are rotated and scaled for illustrative purposes.}
 \label{fig:Pixels}
\end{figure}
Let $f^h$ be an approximation of $f$ in the computational domain such that $f^h_i=f(\ib)$, i.e., the value of the approximate image over the element $T_i$ corresponds to the value of the image at pixel $\ib$.
A \emph{patch} \( \Px{\ib}{f} \)  is a sub-image of \(f^h\) around pixel \(\ib\) given by
\[
\Px{\ib}{f} ({\boldsymbol t}) = f^h( \ib + {\boldsymbol t}), \qquad \forall {\boldsymbol t} \in [-\rho:\rho]^2,
\]
where the interval \([a:b]\) denotes the closed interval of integers from \(a\) to \(b\).

We refer to the sum of the image values within a patch as the \emph{patch measure} and denote it by \(\px{\ib}{f}\). Notice that a patch will have exactly \((2\rho+1)^2 =: | \mathcal{P}|\) pixels. We approximate the value of the kernel in \eqref{eq:w-kernel} at points corresponding to pixels $(\ib,\jb)$ as follows:
\begin{equation}
\begin{aligned}
\label{NLM-approx-FEM}
\gamma^h_{\ib,\jb} &\approx
\exp\Bigg\{
-w \bigg(\px{\ib}{f^2} + \px{\jb}{f^2}
		-2\sum_{{\boldsymbol t} \in [-\rho:\rho]^2}  \Px{\ib}{f} ({\boldsymbol t}) \circ \Px{\jb}{f} ({\boldsymbol t}) \bigg)  \Bigg\}\\
	&\hspace{18em}
	\chi\big(\jb\in B_\veps(\ib) \cap \gamma^h_{\ib,\jb} > \iota \big).
\end{aligned}
\end{equation}
This serves as an approximation of \(\gamma^h_{i,j}\) in \eqref{sys:nonlocal-by-triangles}, where elements are associated with the corresponding pixels.

As noted in \cite{Sharifymoghaddam2015}, high dissimilarity values between each pair of patches do not provide meaningful information to the resulting image restoration process. Therefore, in \eqref{NLM-approx-FEM} we introduce the threshold parameter \(\iota>0\) that acts as an acceptance tolerance between patches.
Furthermore, we consider large interaction radii, i.e. \(\veps\gg 1\). This constraint induces a multi-banded matrix approximation of the nonlocal operator with \(\veps-1\) bands yielding at most \( (2\veps+1)^2 =: \mathcal{E}\) neighbors per pixel. These two constraints ensure that only close and similar regions of the image are compared and, at the same time, it reduces memory allocation and computational cost. As a consequence, the nonlocal kernel can be evaluated in \( \mathcal{O} \Big( | \mathcal{P}| \big[ NM(1+\mathcal{E}) - \mathcal{E} \big] \Big) \) operations. To see this, first notice that by symmetry we only need the lower triangular part of the matrix. There, \( \mathcal{O} \big( | \mathcal{P}| NM \big)\) operations are needed to compute the patch measure of the squared values of each patch. Now, if \(\ib > \mathcal{E} \), then the Hadamard product and sum between each pair of patches takes \( \mathcal{O} \big( | \mathcal{P}| \big) \) operations. After these quantities have been evaluated, \eqref{NLM-approx-FEM} reduces to a comparison and an exponential of a product between scalars which take \( O(1) \). As we only need to compare \( \mathcal{E}/2 \) neighbors, we get \( \mathcal{O} \big( | \mathcal{P}| \mathcal{E} /2\big) \). Thus, the number of operations needed to compute the lower part \( \ib > \mathcal{E} \) is \( \mathcal{O} \big( | \mathcal{P}| (NM - \mathcal{E} ) \mathcal{E}/2 \big) \). For the case \( \ib \leq \mathcal{E} \), we proceed as before and note that the number of neighbors increases with each row. This takes \( \mathcal{O} \big( | \mathcal{P}| \mathcal{E} (\mathcal{E} + 1)/2 \big) \).

\subsubsection{\color{black}Optimization algorithm}\label{sec:PLMTR-algorithm}
The reduced objective functionals \eqref{red-funct-parameter} and \eqref{red-funct-weight} are not necessarily convex since equations \eqref{eq:discretized-os} and \eqref{eq:discretized-os-weights} are nonlinear in terms of {\( (u,\lambda) \) and \( (u,w) \)}, respectively. Thus, we resort to 
the projected trust-region algorithm developed in \cite{Yuan2014} for the solution of general nonlinear box-constrained optimization problems involving limited memory BFGS matrices.
\subsection{Experimental results}\label{sec:implementation-results}

We present the results of the bilevel optimization with respect to \(\lambda\) and \(w\) using the modified nonlocal means kernel. The results are organized as follows: for each test we report a figure and a table. The figure displays the clean image \(u_T\) from a database, four noisy images, (a)--(d), and the corresponding (optimal) denoised images, (e)--(h).
Values of the Structural Similarity Index (SSIM), which measure the similarity of the recovered image against \(u_T\), are also included (rounded up to two digits).
In the table we report optimal SSIM values and output parameters of the optimization.

For our computations, we use images from the USC-SIPI Image Database and the FVC2000 Database, which are padded with a border of width \(\veps\), in order to deal with information in \(\Omgi\). For each image, a sample of four noisy images is obtained by adding different levels of Gaussian noise with standard deviation \(\sigma\); that is \( f = u_T + \eta\) with \( \eta \sim \mathcal{N}(0,\sigma^2)\). The values of \(\sigma\) are taken according to Table \ref{Num-Parameters}.
We use the constant patch radius \(\rho=5\), i.e., each patch contains \(121\) pixels and we set the interaction radius \(\veps\) so that there are at most \(5 \min\{N,M\}\) neighbors per pixel. The problem dimension $N$ and $M$ will be specified below for each experiment.
\begin{table}[h]
\caption{Parameters associated to noisy data}
\begin{center}
\fontsize{9.5}{8}\selectfont
\setlength{\tabcolsep}{4.5pt}
\def\arraystretch{2.5}
\sisetup{table-format = 1.3e2, table-number-alignment = center}
\begin{tabular}{l c c c c}
\toprule
	\bf Sample &	(a) & (b) & (c) & (d)
	\\ \hline
	\bf $\sigma^2$	& \num{e1.5} & \num{e2.} & \num{e2.5} & \num{e3.}
	\\
	\bf Filtering $\delta$	& 
	\num{3e1} & \num{e2} & \num{3.16e2} & \num{3.33e2}
	\\
\bottomrule
\\[-1em]
\end{tabular}
\label{Num-Parameters}
\end{center}
\end{table}

\subsubsection{Optimizing the fidelity parameter \(\lambda\)}
We consider both the case of constant \(\lambda\in[0,b]\) and space-dependent \(\lambda\in\mcC = \big\{\lambda\in H^1(\Omega):b\geq \lambda(\xb) \geq 0\big\}\). In the former case, the upper bound \(b\) is set to \(\num{e5}\), while in the latter it is set to \(255\). The acceptance tolerance is set to \(\iota = 10^{-9}\). Recall that in this case we set $w(\tb)=\delta^{-2}$; values for each image are reported in Table \ref{Num-Parameters}.
\paragraph{Constant parameter}
We initialize the TR algorithm with \(\lambda_0 = 100\) and we note that, for $\lambda\in[0,\Lambda]$, the gradient of the reduced cost functional \eqref{red-funct-parameter} reduces  to \(\nabla j(\lambda) = (\ub-\fb) \cdot \pb\).

The results are displayed in Figure \ref{Scalar-opt-imgs} and Table \ref{Scalar-opt}. In the latter we report, for each clean image and its corresponding noisy sample, the optimal \(\lambda\), its SSIM value, the number of iterations of the TR algorithm, 
and the dimensions of the image. 
From Figure \ref{Scalar-opt-imgs} we note that, after the optimization, there is a significant increase in the SSIM values. Moreover, as expected, the nonlocal means kernel allows denoising of each sample while preserving textures (see, e.g., \cite{gilboa2007nonlocal}). Hence, discontinuities are preserved and restored. 
Furthermore, in Table \ref{Scalar-opt} it is noticeable that the the best solution found for each noisy image is located in the interior of the convex set.

We also note that, at each iteration, the objective function decreases monotonically and the radius of the trust-region decreases around the solution.

\captionsetup[sub]{font=footnotesize, justification=centering}
\begin{figure}[H]
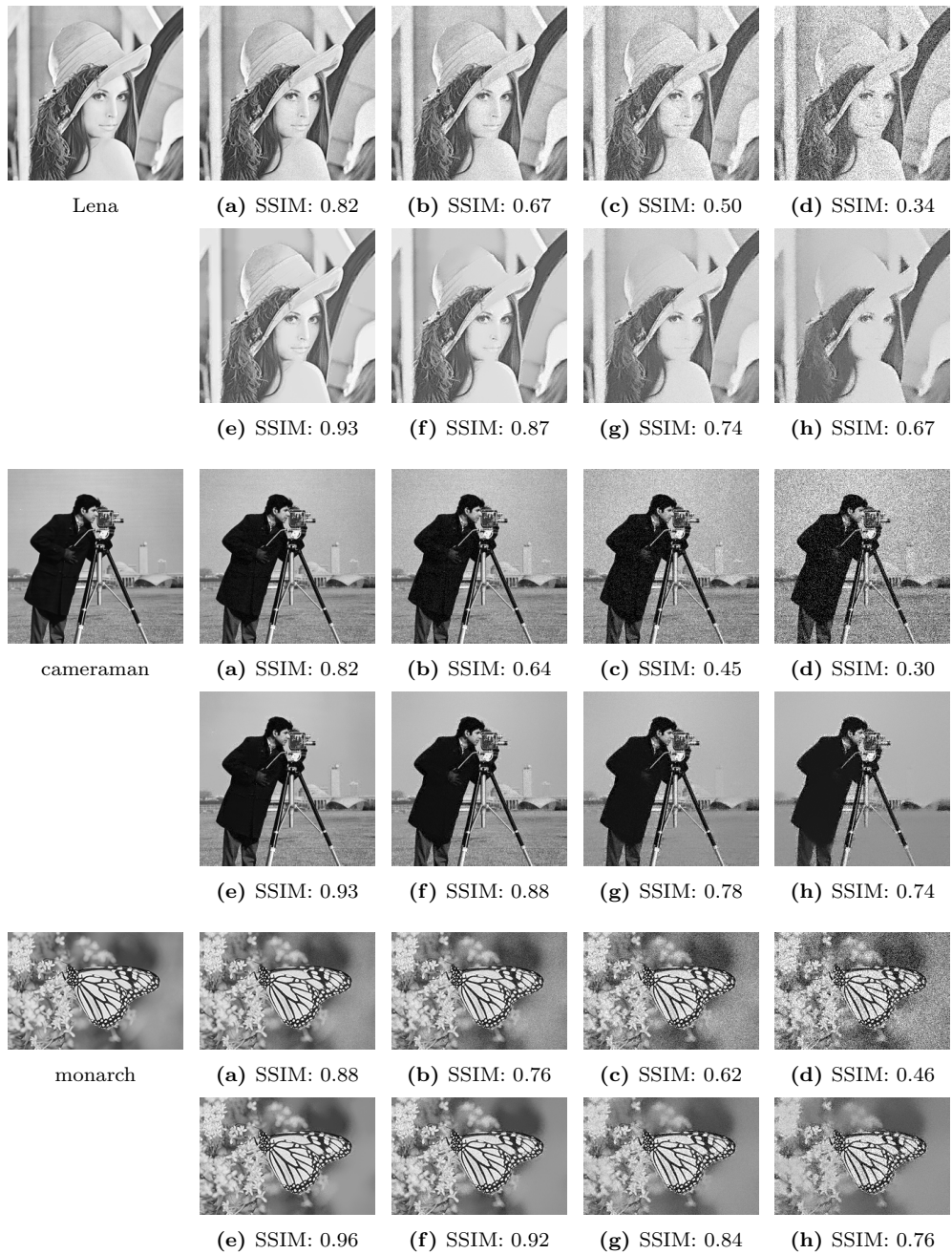

	\centering
	\picb{Lena}{0.185}{1-lena}{lena-bn}	\pic{0.82}{0.185}{1-lena}{1-lena-n}
	\pic{0.67}{0.185}{1-lena}{2-lena-n}	\pic{0.50}{0.185}{1-lena}{3-lena-n}	\pic{0.34}{0.185}{1-lena}{4-lena-n}
	\\[-0.7em]
	\picc{}{0.185}		
	\picf{0.93}{0.185}{1-lena}{Scalar}{1-lena-d}		\picf{0.87}{0.185}{1-lena}{Scalar}{2-lena-d}
	\picf{0.74}{0.185}{1-lena}{Scalar}{3-lena-d}		\picf{0.67}{0.185}{1-lena}{Scalar}{4-lena-d}
	\addtocounter{subfigure}{-8}
	\picb{cameraman}{0.185}{4-cameraman}{cameraman-bn}
	\pic{0.82}{0.185}{4-cameraman}{1-cameraman-n}	\pic{0.64}{0.185}{4-cameraman}{2-cameraman-n}
	\pic{0.45}{0.185}{4-cameraman}{3-cameraman-n}	\pic{0.30}{0.185}{4-cameraman}{4-cameraman-n}
	\\[-0.7em]
	\picc{}{0.185}		\picf{0.93}{0.185}{4-cameraman}{Scalar}{1-cameraman-d}	\picf{0.88}{0.185}{4-cameraman}{Scalar}{2-cameraman-d}
					\picf{0.78}{0.185}{4-cameraman}{Scalar}{3-cameraman-d}  \picf{0.74}{0.185}{4-cameraman}{Scalar}{4-cameraman-d}
	\addtocounter{subfigure}{-8}
	\picb{monarch}{0.185}{5-monarch}{monarch-bn}
	\pic{0.88}{0.185}{5-monarch}{1-monarch-n}		\pic{0.76}{0.185}{5-monarch}{2-monarch-n}
	\pic{0.62}{0.185}{5-monarch}{3-monarch-n}		\pic{0.46}{0.185}{5-monarch}{4-monarch-n}
	\\[-0.7em]
	\picc{}{0.185}		\picf{0.96}{0.185}{5-monarch}{Scalar}{1-monarch-d}
					\picf{0.92}{0.185}{5-monarch}{Scalar}{2-monarch-d}
					\picf{0.84}{0.185}{5-monarch}{Scalar}{3-monarch-d}
					\picf{0.76}{0.185}{5-monarch}{Scalar}{4-monarch-d}
	\caption{Resulting images of scalar parameter optimization
	}
	\label{Scalar-opt-imgs}
\end{figure}

\newpage

\begin{center}
\fontsize{9.5}{8}\selectfont
\setlength{\tabcolsep}{4.5pt}
\def\arraystretch{1.5}
\begin{longtable}{c c  c  c c c  c  c  c  c  c  c c  c c}
\caption{Results of scalar optimization}\\
\toprule
	&\multirow{2}{*}{\bf Sample} &\multirow{2}{*}{\bf Best $\lambda$} & \multirow{2}{*}{\bf SSIM} &  \multirow{2}{*}{\bf PSNR}	& \bf Iteration	    
	& \multirow{2}{*}{$(N,M)$}
	\\
	&&&   && \bf Count 
\\
\hline
\multirow{4}{*}{\rotatebox{90}{Lena}}
& (a)	& 0.950959118335018 	& 0.9279 & 34.84 	& 26	    
& \multirow{4}{*}{$(256,256)$}	\\
& (b)	& 36.30572679235037 	& 0.8734 & 30.78 	& 7	
\\
& (c)	& 374.3532501043157	& 0.7443 & 26.85 	& 17    
\\
& (d)	& 165.5679531767198 	& 0.6725    & 23.87 	& 15    
\\
\hline
\multirow{4}{*}{\rotatebox{90}{cameraman}}
& (a)	& 45.40495436349488	& 0.9291    & 35.13	& 7    
& \multirow{4}{*}{$(256,256)$}	\\
& (b)	& 197.5505793039822	& 0.8802	& 30.97 & 21    
\\
& (c)	& 325.0933609572109 & 0.7847	& 27.73 &13    
\\
& (d)	& 80.29128369206099	& 0.7374	& 24.66 & 4    
\\
\hline
\multirow{4}{*}{\rotatebox{90}{monarch}}
& (a)	& 8.67751855730995	& 0.9592    & 34.30	& 19     
& \multirow{4}{*}{$(256,171)$}	\\
& (b)	& 70.93045929379818	& 0.9187	& 29.67 & 7    
\\
& (c)	& 274.0138712304279 & 0.8356	& 25.90 & 19     
\\
& (d)	& 126.1323535975534	& 0.7551	& 22.24 & 9    
\\
\bottomrule
\\[-1em]
\label{Scalar-opt}
\end{longtable}
\end{center}


\paragraph{Spatially-dependent parameter}

The optimization with respect to a spatially-de\-pendent \(\lambda\) is a large scale nonconvex problem. Thus, to prevent stagnation in regions far from local minima, we restart the optimization in the following two cases \cite{Maggiar2018,Cartis2018}:
\begin{enumerate}
	\item {\color{black}The trust region radius becomes smaller than a prescribed tolerance. In such case, we reset its value and} continue iterating if there is a decrease in the objective function. This is done in order to prevent algorithm to halt at a non-stationary point, whenever the trust region radius decreases too quickly.

	\item 
	{\color{black}The curvature condition is too close to zero. In this case, then all the previously stored pairs of derivates and evaluation points are removed and rebuilt from scratch.} This prevents the occurrence of ill-conditioned updates. 
\end{enumerate}
Moreover, after each successful update of the limited memory algorithm, we modify the L-BFGS initialization parameter, 
using Dener and Munson scalar initialisation approach \cite{Dener-2019}. 
We set the maximum number of iterations to \(\num{e2}\) initializing with the constant candidate \(\lambda_0 = 200\).

The results are displayed in Figure \ref{Spatial-opt-imgs} and Table \ref{Spatial-opt}.
In addition to the noisy sample and its corresponding set of solutions to each image \(u_T\) of the database, we also include a third row of images displaying the optimal $\lambda(\xb)$.
In the table we report the optimal SSIM value
and the dimensions of the image.

In Figure \ref{Spatial-opt-imgs}, we note that the optimal SSIM is much higher than the one associated with the noisy image and that there is a significant improvement compared to results obtained with a constant $\lambda$. We also observe that the optimal parameter is able to catch discontinuities and noise, see in particular (c) and (d).

\vfill
\newpage
	\begin{center}
	\fontsize{9.5}{8}\selectfont
	\setlength{\tabcolsep}{4.5pt}
	\def\arraystretch{1.5}
	\begin{longtable}{c c  c  c c c  c  c  c  c  c  c c  c c}
	\caption{Results of optimization with respect to $\lambda(\xb)$.}\\
	\toprule
		&\multirow{2}{*}{\bf Sample} &\multirow{2}{*}{\bf SSIM}	&\multirow{2}{*}{\bf PSNR} & \bf Iteration	 
		& \multirow{2}{*}{$(N,M)$}
		\\
		&&&&   \bf Count 
	\\[0.2em]
	\hline
	\multirow{4}{*}{\rotatebox{90}{Lena}}
	& (a)	& 0.9201    & 34.68 	& 77 	    
	& \multirow{4}{*}{$(256,256)$}	\\
	& (b)	& 0.9093    & 32.18 & 101 		
	\\
	& (c)	& 0.8051 	& 27.95 & 1000 	    
	\\
	& (d)	& 0.7702 	& 26.27 & 1000 	    
	\\
	\hline
	\multirow{4}{*}{\rotatebox{90}{cameraman}}
	& (a)	& 0.9366	& 35.27 & 36	    
	& \multirow{4}{*}{$(256,256)$}	\\
	& (b)	& 0.9115	& 32.25 & 84		    
	\\
	& (c)	& 0.8435	& 29.02 & 300	    
	\\
	& (d)	& 0.7864	& 26.84 & 162	    
	\\
	\hline
	\multirow{4}{*}{\rotatebox{90}{monarch}}
	& (a)	& 0.9656 & 34.57 & 146	
	& \multirow{4}{*}{$(256,171)$}	\\
	& (b)	& 0.9393	& 30.30 & 65		
	\\
	& (c)	& 0.8918 	& 27.00 & 192		
	\\
	& (d)	& 0.8342	& 23.75 & 431	    
	\\
	\bottomrule
	\\[-1em]
	\label{Spatial-opt}
	\end{longtable}
	\end{center}

\captionsetup[sub]{font=footnotesize, justification=centering}
	\begin{figure}[H]
		\centering
		\picb{Lena}{0.185}{1-lena}{lena-bn}	\pic{0.82}{0.185}{1-lena}{1-lena-n}
		\pic{0.67}{0.185}{1-lena}{2-lena-n}	\pic{0.50}{0.185}{1-lena}{3-lena-n}	\pic{0.34}{0.185}{1-lena}{4-lena-n}
		\\[-0.7em]
		\addtocounter{subfigure}{-4}
		\picc{}{0.185}		\picf{0.92}{0.185}{1-lena}{Spatial}{1-lena-d}		\picf{0.91}{0.185}{1-lena}{Spatial}{2-lena-d}
		\picf{0.81}{0.185}{1-lena}{Spatial}{3-lena-d}		\picf{0.77}{0.185}{1-lena}{Spatial}{4-lena-d}
		\\[-0.7em]
		\picc{}{0.185}		\picl{0.185}{1-lena}{Spatial}{1-lena-l}		\picl{0.185}{1-lena}{Spatial}{2-lena-l}
						\picl{0.185}{1-lena}{Spatial}{3-lena-l}		\picl{0.185}{1-lena}{Spatial}{4-lena-l}
		\\[-0.5em]
		\caption{Resulting images of spatial parameter optimization}
	\end{figure}
	
	\begin{figure}[H]
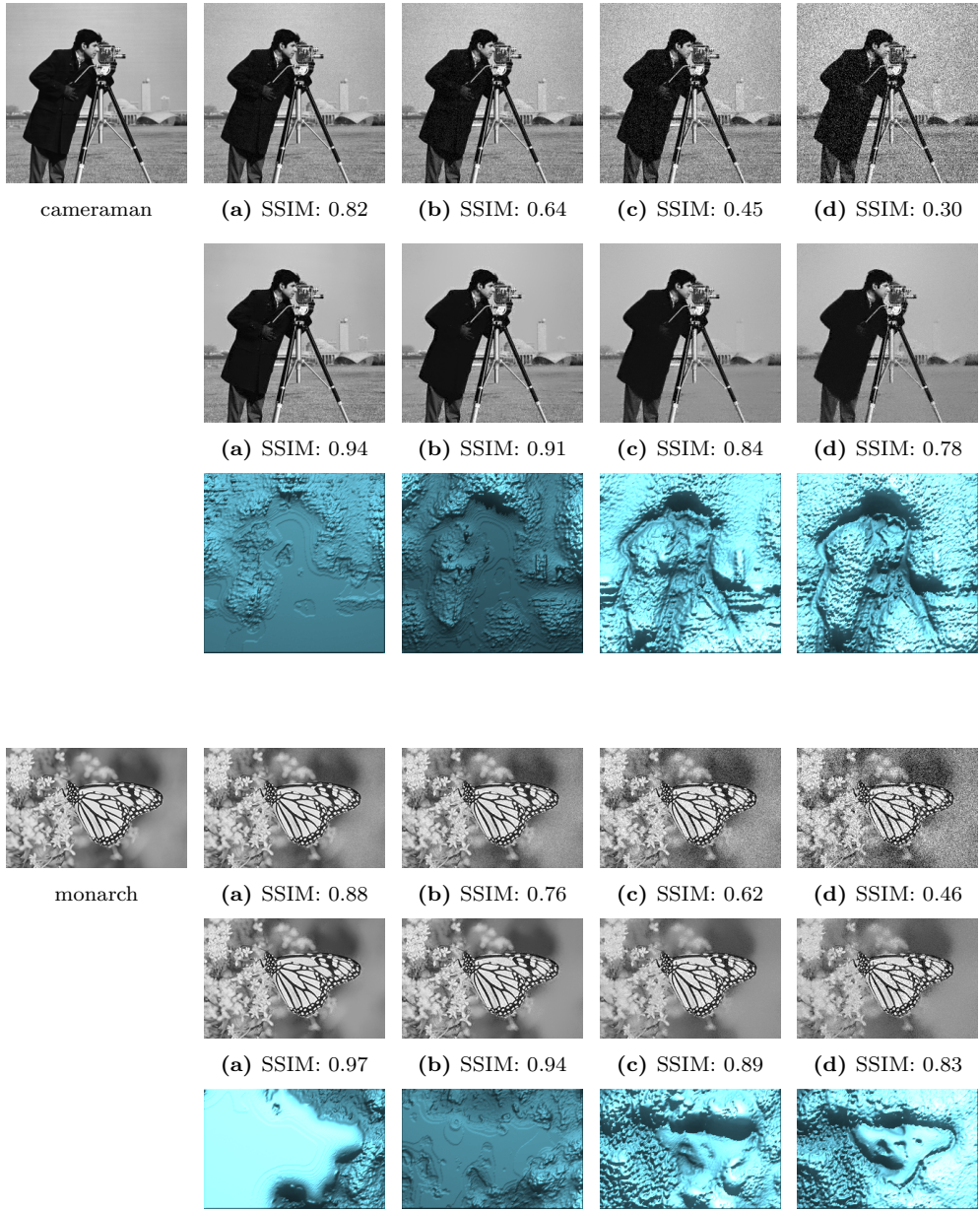
\ContinuedFloat 
		\centering
		\addtocounter{subfigure}{-4}
		\picb{cameraman}{0.185}{4-cameraman}{cameraman-bn}
		\pic{0.82}{0.185}{4-cameraman}{1-cameraman-n}	\pic{0.64}{0.185}{4-cameraman}{2-cameraman-n}
		\pic{0.45}{0.185}{4-cameraman}{3-cameraman-n}	\pic{0.30}{0.185}{4-cameraman}{4-cameraman-n}
		\\[-0.3em]
		\addtocounter{subfigure}{-4}
		\picc{}{0.185}		\picf{0.94}{0.185}{4-cameraman}{Spatial}{1-cameraman-d}	\picf{0.91}{0.185}{4-cameraman}{Spatial}{2-cameraman-d}
		\picf{0.84}{0.185}{4-cameraman}{Spatial}{3-cameraman-d}  \picf{0.78}{0.185}{4-cameraman}{Spatial}{4-cameraman-d}
		\\[-0.7em]
		\picc{}{0.185}		\picl{0.185}{4-cameraman}{Spatial}{1-cameraman-l}	\picl{0.185}{4-cameraman}{Spatial}{2-cameraman-l}
		\picl{0.185}{4-cameraman}{Spatial}{3-cameraman-l}  \picl{0.185}{4-cameraman}{Spatial}{4-cameraman-l}
		\\[1em]
		\addtocounter{subfigure}{-4}
		\picb{monarch}{0.185}{5-monarch}{monarch-bn}
		\pic{0.88}{0.185}{5-monarch}{1-monarch-n}		\pic{0.76}{0.185}{5-monarch}{2-monarch-n}
		\pic{0.62}{0.185}{5-monarch}{3-monarch-n}		\pic{0.46}{0.185}{5-monarch}{4-monarch-n}
		\\[-0.7em]
		\addtocounter{subfigure}{-4}
		\picc{}{0.185}		\picf{0.97}{0.185}{5-monarch}{Spatial}{1-monarch-d}
		\picf{0.94}{0.185}{5-monarch}{Spatial}{2-monarch-d}
		\picf{0.89}{0.185}{5-monarch}{Spatial}{3-monarch-d}
		\picf{0.83}{0.185}{5-monarch}{Spatial}{4-monarch-d}
		\\[-0.7em]
		\picc{}{0.185}		\picl{0.185}{5-monarch}{Spatial}{1-monarch-l}
						\picl{0.185}{5-monarch}{Spatial}{2-monarch-l}
						\picl{0.185}{5-monarch}{Spatial}{3-monarch-l}
						\picl{0.185}{5-monarch}{Spatial}{4-monarch-l}
		\\[-0.1em]
		\caption{Resulting images of spatial parameter optimization
		}
		\label{Spatial-opt-imgs}
	\end{figure}

\newpage

\paragraph{Training set}

We consider a batch learning approach: we are interested in learning the fidelity constant \(\lambda\) from several images with the same noise level by solving the following coupled optimization problem:
\begin{subequations}\label{eq:bilevel_tr1_prob}
\begin{equation}\label{eq:tr1_opt_prob}
	\min_{0 \leq \lambda \leq b} \; L(U)
\end{equation}

\vspace*{-.6cm}

\begin{equation}\label{eq:tr1_rest_denoising}
{\rm s.t.}	\quad \mu (u_i,\psi)_{V_i} + \lambda (u_i-f_i,\psi)_{0,\Omg}=0, \quad \forall \psi \in V_{i}, i \in I,
\end{equation}
where \(U = \{u_i\}_{i \in I}\) is a set of reconstructed images from a noisy sample \( F = \{f_i\}_{i \in I} \) and \(L\) is a generalization of the loss function \(\ell\), defined as
\begin{equation}\label{eq:tr_loss_fun}
	L(U) = \dfrac{1}{|I|} \sum_{i\in I} \ell (u_i) = \dfrac{1}{2 |I|} \sum_{i\in I} \|u_i-u_i^T\|^2_{0,\Omg},
\end{equation}
being \( u_T = \{u_i^T\}_{i\in I} \) a  set of clean images.
Finally, \(V_i\) is the function space associated with the kernel generated by the ground-truth image \(u^T_i\).
\end{subequations}

Following the analysis carried out for problem \eqref{eq:opt_prob_ESC}, we can further get a system of adjoint equations and a reduced derivative, resulting in the following optimality system
\begin{subequations}\label{opt_system:tr_scalar_lambda}
\begin{align}\label{opt_system:tr_scalar_lambda_a}
	\mu (u_i,\psi)_{V_{i}} + \lambda(u_i-f_i,\psi)_{0,\Omg} &=0, & \forall \psi &\in V_{i}, i \in I,
	\\
	\mu (p_i,\phi)_{V_{i}} + \lambda(p_i,\phi)_{0,\Omg} &= \dfrac{1}{|I|}(u_i^T - u_i, \phi)_{0,\Omg}, & \forall \phi &\in V_{i},	\label{opt_system:tr_scalar_lambda_b}
	\\
	P_{[0,b]} \left(\lambda -c \sum_{i\in I} (u_i-f_i,p_i) \right) &= \lambda, & \forall c&>0.
  \label{opt_system:tr_scalar_lambda_c}
\end{align}
\end{subequations}
The training set of clean and noisy images is constructed as follows: we select ten images \(\{u_i^T\}_{i\in I}\) from the FVC2000 Database and resize them down to \(203 \times 190\) pixels. Then, pixelwise, we add Gaussian noise of variance \(\sigma^2 = \num{e3}\) to obtain the noisy data \( \{f_i \}_{i\in I} \).

We initialize the TR algorithm with \(\lambda_0 = \num{e-1}\), \(b=\num{e10}\), and set the weight of the kernel to \(w = \num{9.3e-5}\). The results are displayed in Figure \ref{Training-opt-imgs} and Table \ref{Training-opt}, respectively. In the latter, we report the SSIM value of the reconstruction for each image.
After 20 iterations of the algorithm, the optimal value of \(\lambda\) was 10.6393411229.

In Figure \ref{Training-opt-imgs} we note an increase in the SSIM values of the denoised images, compared to the clean ones. 
%
Similarly, in Figure \ref{Test-opt-img} we show the outcome of the denoising algorithm (lower-level problem) for one noisy image that does not belong to the training set. The lower-level problem is solved in correspondence of the optimal (trained) \(\lambda\), and there is also an increment in the SSIM value.

\begin{figure}[H]
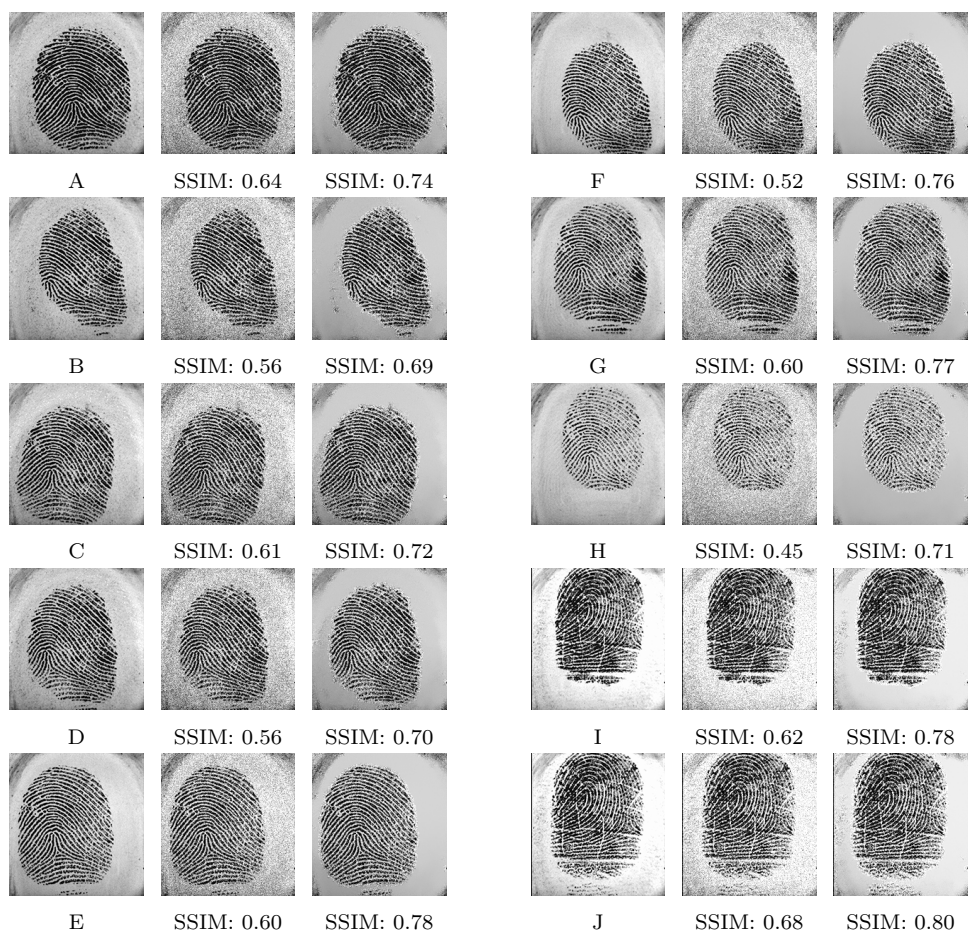

		\centering
		\picb{A}{0.135}{Fingerprints}{a}
				\picb{SSIM: 0.64}{0.135}{Fingerprints/Noisy}{a-n}	\picb{SSIM: 0.74}{0.135}{Fingerprints/Denoised}{a-d}
		\picc{}{0.05}
		\picb{F}{0.135}{Fingerprints}{f}
				\picb{SSIM: 0.52}{0.135}{Fingerprints/Noisy}{f-n}	\picb{SSIM: 0.76}{0.135}{Fingerprints/Denoised}{f-d}
		\\[-1em]
		\picb{B}{0.135}{Fingerprints}{b}
				\picb{SSIM: 0.56}{0.135}{Fingerprints/Noisy}{b-n}	\picb{SSIM: 0.69}{0.135}{Fingerprints/Denoised}{b-d}
		\picc{}{0.05}
		\picb{G}{0.135}{Fingerprints}{g}
				\picb{SSIM: 0.60}{0.135}{Fingerprints/Noisy}{g-n}	\picb{SSIM: 0.77}{0.135}{Fingerprints/Denoised}{g-d}
		\\[-1em]
		\picb{C}{0.135}{Fingerprints}{c}
				\picb{SSIM: 0.61}{0.135}{Fingerprints/Noisy}{c-n}	\picb{SSIM: 0.72}{0.135}{Fingerprints/Denoised}{c-d}
		\picc{}{0.05}
		\picb{H}{0.135}{Fingerprints}{h}
				\picb{SSIM: 0.45}{0.135}{Fingerprints/Noisy}{h-n}	\picb{SSIM: 0.71}{0.135}{Fingerprints/Denoised}{h-d}
		\\[-1em]
		\picb{D}{0.135}{Fingerprints}{d}
				\picb{SSIM: 0.56}{0.135}{Fingerprints/Noisy}{d-n}	\picb{SSIM: 0.70}{0.135}{Fingerprints/Denoised}{d-d}
		\picc{}{0.05}
		\picb{I}{0.135}{Fingerprints}{i}
				 \picb{SSIM: 0.62}{0.135}{Fingerprints/Noisy}{i-n}	\picb{SSIM: 0.78}{0.135}{Fingerprints/Denoised}{i-d}
		\\[-1em]
		\picb{E}{0.135}{Fingerprints}{e}
			 	\picb{SSIM: 0.60}{0.135}{Fingerprints/Noisy}{e-n}	\picb{SSIM: 0.78}{0.135}{Fingerprints/Denoised}{e-d}
		\picc{}{0.05}
		\picb{J}{0.135}{Fingerprints}{j}
				\picb{SSIM: 0.68}{0.135}{Fingerprints/Noisy}{j-n}	\picb{SSIM: 0.80}{0.135}{Fingerprints/Denoised}{j-d}
		\\[-1em]
		\caption{Resulting images of scalar parameter training}
		\label{Training-opt-imgs}
	\end{figure}

    \begin{figure}[H]
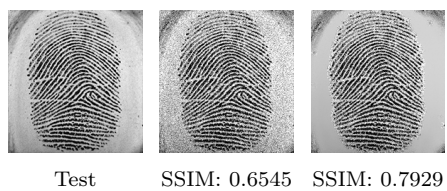

		\centering
		\picb{Test}{0.135}{Fingerprints/Test}{t_0}
				\picb{SSIM: 0.6545}{0.135}{Fingerprints/Test}{t_0-n}	\picb{SSIM: 0.7929}{0.135}{Fingerprints/Test}{t_0-d}
    \\[-1em]
		\caption{Validation image with trained parameter}
		\label{Test-opt-img}
	\end{figure}
	
\newpage
\begin{center}
\fontsize{9.5}{8}\selectfont
\setlength{\tabcolsep}{4.5pt}
\def\arraystretch{1.5}
\begin{longtable}{c c  c  c c c  c  c  c  c  c  c c  c c}
\caption{Results of batch training}
\\
\toprule
	\bf Sample & \bf SSIM & \bf PSNR	&& \bf Sample & \bf SSIM & \bf PSNR
\\
\hline
A	& 	0.7447 & 20.88	&&	F	& 	0.7632	& 21.40	\\
B	&  	0.6885 & 21.38	&&	G	& 	0.7662  & 20.71		\\
C	&  	0.7185 & 20.78	&& 	H	& 	0.7108 	& 21.55	\\
D	&  	0.7047 & 21.06	&&	I	& 	0.7801  & 20.90		\\
E	& 	0.7810 & 20.82	&&	J	& 	0.7981  & 20.49
\\
\bottomrule
\\[-1em]
\label{Training-opt}
\end{longtable}
\end{center}

\subsection{Optimizing with respect to the weight \(w\)}
\label{sec:implementation-weights}
We solve problem \eqref{eq:weight-opt} with \(w\in \mcC = [0,W]\). Note that every evaluation of the reduced objective functional \eqref{red-funct-weight} requires the numerical solution of \eqref{eq:lower-optimality-w}, and hence requires updating \(\gamma_w\). However, by definition, we have that \( \gamma_{w_1} = \gamma_{w_0}^{\nicefrac{w_1}{w_0}} \), which provides a fast way to get a new kernel for any \(w_1, w_0 \in \mcC\).

Additionally, the gradient of the reduced objective functional \eqref{red-funct-weight} requires the computation of the linearized kernel \(\widehat\gamma_{w}^h\), see \eqref{ec:disc-weight-derivative}. According to \eqref{NLM-approx-FEM}, we have
\begin{equation}\label{ec:disc-weight-der-expansion}
\begin{aligned}
\widehat\gamma_{w,\ib,\jb}^h 
& = \widetilde\gamma_{(1),\ib,\jb}^h \\
& = \gamma_{w,\ib,\jb}^h \, \cdot \bigg(    {-}\px{\ib}{f^2} - \px{\jb}{f^2}
  + 2\sum_{{\boldsymbol t} \in [-\rho:\rho]^2}  
    \Px{\ib}{f} ({\boldsymbol t}) \circ \Px{\jb}{f} ({\boldsymbol t}) \bigg)
\end{aligned}
\end{equation}
for all pixels \(\ib,\jb \in \mathcal{T}^h\).
Letting \( \widecheck{\gamma}_{i,j}^h = {-}\px{\ib}{f^2} - \px{\jb}{f^2}
  + 2\sum_{{\boldsymbol t} \in [-\rho:\rho]^2}  
    \Px{\ib}{f} ({\boldsymbol t}) \circ \Px{\jb}{f} ({\boldsymbol t})\),
\( \widehat\gamma_{w}^h \) can be easily computed by the Hadamard product \( \gamma_{w}^h \circ \widecheck{\gamma}^h\). 
Furthermore, \(\widecheck{\gamma}^h\) depends on the noisy image \(f\) only. Thus, it is computed once and for all and we have \(\gamma_w^h = \exp\{ -w \cdot \widecheck{\gamma}^h \}\).

As numerically, the exponential function has a limited exponent range which prevents the effects of underflowing and overflowing, care has to be taken whenever choosing \(W\) and \(\iota\). 
Considering that the entries of \(\gamma_w^h\) are in \([0,1]\), here we focus on avoiding underflow. This numerical condition occurs for images with high levels of noise, i.e., patches are highly dissimilar, resulting in a matrix with entries close to \(0\). Hence, if \(W\) is high, then the formula \( \gamma_{w_1} = \gamma_{w_0}^{\nicefrac{w_1}{w_0}} \) can return a constant matrix with no further possible updates. In contrast, if we make \(W\) small, then optimizing in images with low levels of noise, i.e., close-to-one patch distance, will result in an underestimation of the optimal value for \(w\). This comes as \(W\) can be taken as low such that it is accepted as optimal, whereas the best image reconstruction could require \(w\geq W\).
In order to avoid this behavior, we set \( W = K \max\{ \nicefrac{300}{\max \widecheck\gamma^h} , \nicefrac{5}{\kappa} \times \num{e-5} \} \) with \(K\) given as in Table \ref{Weight-Parameters} and \(\kappa\) is a scaling parameter introduced below. This value is chosen so that whenever the entries of \(\widecheck\gamma^h\) are small due to low levels of noise, cases (a) and (b), then \(w\) can be taken as big as some multiple of \(300\) that avoids underflow; and if the entries of \(\widecheck\gamma^h\) are big due to high levels of noise, cases (c) and (d), then the values of \(w\) will be again limited to avoid a constant matrix. Now, for the acceptance tolerance we set \(\iota = \num{e-10}\) which will be applied once for an initial kernel of parameter \(w_{-1} = \num{e-6}\). This allows us to keep entries that could be deleted whenever \( w > w_{-1}\), yet still remove entries with high dissimilarity values.
Additionally, as in practice the numerical range \(\mcC\) is small, we scale the argument of the objective function in order to further avoid cancellation errors whenever reaching a local minimum. For this, we set the scaling parameter \(\kappa = \num{e-6}\).
Finally, we set \(\lambda=\nicefrac 1 2\).

	\begin{table}[h]
		\caption{Parameters for weight optimization}
		\begin{center}
		\fontsize{9.5}{8}\selectfont
		\setlength{\tabcolsep}{4.5pt}
		\def\arraystretch{2.5}
		\sisetup{table-format = 1.3e2, table-number-alignment = center}
		\begin{tabular}{c c c c c}
		\toprule
			\bf Sample &	(a) & (b) & (c) & (d)
			\\ \hline
			\bf $K$ 	& 12 & 9 & 6 & 3
			\\
		\bottomrule
		\\[-1em]
		\end{tabular}
		\label{Weight-Parameters}
		\end{center}
	\end{table}

We initialize $w$ with \(w_0 = \num{1.1e-6}\). 
The corresponding results are presented in Figure \ref{Weight-opt-imgs} and Table \ref{Weight-opt}. For each clean image and its corresponding noisy sample, we report the optimal \(w\), the corresponding SSIM, the number of iterations of the TR algorithm, 
and the dimensions of the image. We again observe a significant increase in the SSIM values and that the optimal parameters are within the interior of the convex set \(\mcC\).

	\begin{figure}[H]
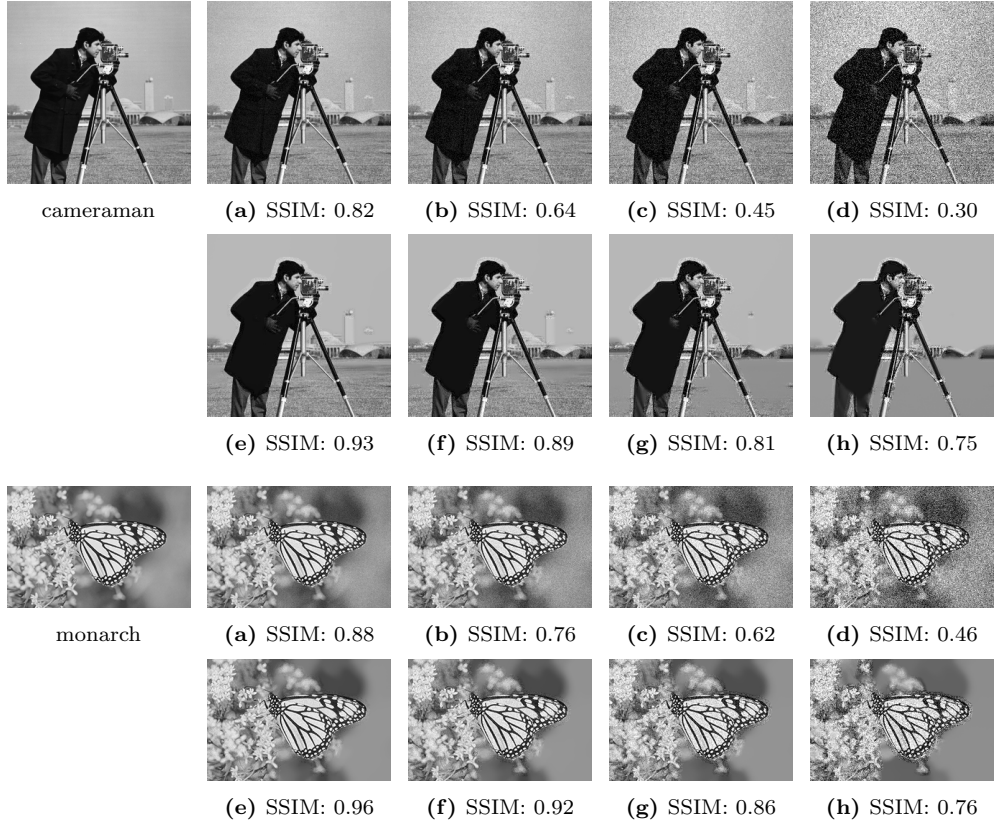
 
		\centering
		\picb{Lena}{0.185}{1-lena}{lena-bn}	\pic{0.82}{0.185}{1-lena}{1-lena-n}
		\pic{0.67}{0.185}{1-lena}{2-lena-n}	\pic{0.50}{0.185}{1-lena}{3-lena-n}	\pic{0.34}{0.185}{1-lena}{4-lena-n}
		\\[-0.7em]
		\picc{}{0.185}		\picf{0.93}{0.185}{1-lena}{Weight}{1-lena-d}	\picf{0.86}{0.185}{1-lena}{Weight}{2-lena-d}
		\picf{0.83}{0.185}{1-lena}{Weight}{3-lena-d}	\picf{0.75}{0.185}{1-lena}{Weight}{4-lena-d}
    \caption{Results for the kernel weight optimization}
	\end{figure}
	
	\newpage
	
	\begin{figure}[H] \ContinuedFloat 
		\addtocounter{subfigure}{-8}
		\picb{cameraman}{0.185}{4-cameraman}{cameraman-bn}
		\pic{0.82}{0.185}{4-cameraman}{1-cameraman-n}	\pic{0.64}{0.185}{4-cameraman}{2-cameraman-n}
		\pic{0.45}{0.185}{4-cameraman}{3-cameraman-n}	\pic{0.30}{0.185}{4-cameraman}{4-cameraman-n}
		\\[-0.7em]
		\picc{}{0.185}		\picf{0.93}{0.185}{4-cameraman}{Weight}{1-cameraman-d}
		\picf{0.89}{0.185}{4-cameraman}{Weight}{2-cameraman-d}
		\picf{0.81}{0.185}{4-cameraman}{Weight}{3-cameraman-d}
		\picf{0.75}{0.185}{4-cameraman}{Weight}{4-cameraman-d}
		\\
		\addtocounter{subfigure}{-8}
		\picb{monarch}{0.185}{5-monarch}{monarch-bn}
		\pic{0.88}{0.185}{5-monarch}{1-monarch-n}		\pic{0.76}{0.185}{5-monarch}{2-monarch-n}
		\pic{0.62}{0.185}{5-monarch}{3-monarch-n}		\pic{0.46}{0.185}{5-monarch}{4-monarch-n}
		\\[-0.7em]
		\picc{}{0.185}		\picf{0.96}{0.185}{5-monarch}{Weight}{1-monarch-d}
		\picf{0.92}{0.185}{5-monarch}{Weight}{2-monarch-d}
		\picf{0.86}{0.185}{5-monarch}{Weight}{3-monarch-d}
		\picf{0.76}{0.185}{5-monarch}{Weight}{4-monarch-d}
		\\
		\caption{Results for the kernel weight optimization}
		\label{Weight-opt-imgs}
	\end{figure}

	\begin{center}
	\fontsize{9.5}{8}\selectfont
	\setlength{\tabcolsep}{4.5pt}
	\def\arraystretch{1.5}
	\begin{longtable}{c c  c  c c c  c  c  c  c  c  c c  c c}
	\caption{Results of weight optimization
		}
	\\
	\toprule
		&\multirow{2}{*}{\bf Sample} &\multirow{2}{*}{\bf Best $w$} & \multirow{2}{*}{\bf SSIM}  & \multirow{2}{*}{\bf PSNR} 	& \bf Iteration	    
		& \multirow{2}{*}{$(N,M)$}
		\\
		&&&&   & \bf Count 
	\\
	\hline
	\multirow{4}{*}{\rotatebox{90}{Lena}}
	& (a)	& 0.0012423963428871366 	& 0.9283 & 34.71 & 54   
	& \multirow{4}{*}{$(256,256)$}	\\
	& (b)	& \num{5.880388893142e-04} 	& 0.8858 & 30.40 & 13   
	\\
	& (c)	& \num{3.219376261113e-04} 	& 0.8266 & 26.86 & 10   
	\\
	& (d)	& \num{1.467533224534e-04} 	& 0.7533 & 23.40 & 11   
	\\
	\hline
	\multirow{4}{*}{\rotatebox{90}{cameraman}}
	& (a)	& \num{6.695785895550e-04} 	& 0.9260	& 31.60 & 14  
	& \multirow{4}{*}{$(256,256)$}	\\
	& (b)	& \num{3.572893186495e-04}	& 0.8864	& 29.39 & 21  
	\\
	& (c)	& \num{1.239828976731e-04} 	& 0.8091	& 27.03 & 15  
	\\
	& (d)	& \num{4.826846435709e-05}	& 0.7459	& 24.70 & 9  
	\\
	\hline
	\multirow{4}{*}{\rotatebox{90}{monarch}}
	& (a)	& \num{9.938714024271e-04}	& 0.9617	& 33.36 & 10 	
	& \multirow{4}{*}{$(256,171)$}	\\
	& (b)	& \num{7.454035518203e-04}	& 0.9245	& 29.37 & 9    
	\\
	& (c)	& \num{4.969357012135e-04}	& 0.8565	& 25.05 & 9     
	\\
	& (d)	& \num{1.830612093756e-04}	& 0.7644	& 20.73 & 16    
	\\
	\bottomrule
	\label{Weight-opt}
	\end{longtable}
	\end{center}

\subsection{Comparison between methods}\label{sec:results-comparison}
Finally, we briefly compare the results obtained after optimizing problems \eqref{eq:opt_prob_ESC}, \eqref{eq:bilevel_reg_prob}, and \eqref{eq:weight-opt}, and compare them with total variation and total generalized variation denoising. For this purpose, we select an image of a zebra and add Gaussian noise with standard deviation $\sigma=\num{e3}$. Each result is displayed in Figure \ref{fig:Comparison}, where the SSIM value of the optimal image is also provided. Moreover, a close-up of each image is plotted, in order to compare the graphical differences of each method.

Visually, it is clear that total variation approaches do not perform as well as most of the nonlocal approaches. The well-know staircasing effect of total variation is present in the scene of the zebra. Moreover, the local approaches present an greater increase in brightness induced by noise. In (c) and (e), it is noticeable that some noise is retained in the zebra's head from underfitting which is corrected in (d).

\begin{figure}[H]
	\centering
	\addtocounter{subfigure}{-4}
	\picb{fprint3}{0.15}{Zebra}{zebra}
	%
	\pic{0.706}{0.15}{Zebra}{4-zebra-tv}
	\picf{0.716}{0.15}{Zebra}{}{4-zebra-tgv}
	\picf{0.728}{0.15}{Zebra}{}{4-zebra-scalar}
	\picf{0.795}{0.15}{Zebra}{}{4-zebra-spatial}
	\picf{0.741}{0.15}{Zebra}{}{4-zebra-weight}
	\\[-0.7em]
	\addtocounter{subfigure}{-5}
	\picb{}{0.15}{Zebra}{zoom-clean}
	\pica{TV}{0.15}{Zebra}{zoom-tv}
	\pica{TGV}{0.15}{Zebra}{zoom-tgv}
	\pica{Scalar \(\lambda\)}{0.15}{Zebra}{zoom-sc}
	\pica{Spatial \(\lambda\)}{0.15}{Zebra}{zoom-sp}
	\pica{Weight \(w\)}{0.15}{Zebra}{zoom-weight}
	\\[-0.3em]
	\caption{Comparison between local and nonlocal denoising methods
	\\[0.5em]
	\scriptsize
	(a) Total Variation denoising, (b) Total Generalized Variation denoising (c) Nonlocal denoising for scalar \(\lambda\), (d) Nonlocal denoising for spatially dependent \(\lambda\), (e) Nonlocal denoising for kernel scalar \(w\).
	}
	\label{fig:Comparison}
\end{figure}




\section{Acknowledgments}
MD was supported by Sandia National Laboratories (SNL), SNL is a multimission laboratory managed and operated by National Technology and Engineering Solutions of Sandia, LLC., a wholly owned subsidiary of Honeywell International, Inc., for the U.S. Department of Energys National Nuclear Security Administration contract number DE-NA0003525. This paper describes objective technical results and analysis. Any subjective views or opinions that might be expressed in the paper do not necessarily represent the views of the U.S. Department of Energy or the United States Government. Sandia report number SAND2020-4564. This material is based upon work supported by the U.S. Department of Energy, Office of Science, Office of Advanced Scientific Computing Research under Award Number DE-SC-0000230927 and the Collaboratory on Mathematics and Physics-Informed Learning Machines for Multiscale and Multiphysics Problems (PhILMs) project.

\newpage
\bibliographystyle{plain}
\bibliography{bibliography}

\begin{thebibliography}{10}

\bibitem{Alali2012}
B.~Alali and R.~Lipton.
\newblock Multiscale dynamics of heterogeneous media in the peridynamic
  formulation.
\newblock {\em Journal of Elasticity}, 106(1):71--103, 2012.

\bibitem{Antil2017imagingSpectral}
H.~Antil and S.~Bartels.
\newblock Spectral approximation of fractional {PDE}s in image processing and
  phase field modeling.
\newblock {\em Computational Methods in Applied Mathematics}, 17(4):661--678,
  2017.

\bibitem{Askari2008}
E.~Askari.
\newblock Peridynamics for multiscale materials modeling.
\newblock {\em Journal of Physics: Conference Series, IOP Publishing},
  125(1):649--654, 2008.

\bibitem{Bates1999}
P.W. Bates and A.~Chmaj.
\newblock An integrodifferential model for phase transitions: Stationary
  solutions in higher space dimensions.
\newblock {\em J. Statist. Phys.}, 95:1119--1139, 1999.

\bibitem{Benson2000}
D.A. Benson, S.W. Wheatcraft, and M.M. Meerschaert.
\newblock Application of a fractional advection-dispersion equation.
\newblock {\em Water Resources Research}, 36(6):1403--1412, 2000.

\bibitem{buades2005review}
Antoni Buades, Bartomeu Coll, and Jean-Michel Morel.
\newblock A review of image denoising algorithms, with a new one.
\newblock {\em Multiscale Modeling \& Simulation}, 4(2):490--530, 2005.

\bibitem{buades2011self}
Antoni Buades, Bartomeu Coll, and Jean-Michel Morel.
\newblock Self-similarity-based image denoising.
\newblock {\em Communications of the ACM}, 54(5):109--117, 2011.

\bibitem{Burch2014}
N.~Burch, M.~D'Elia, and R.~Lehoucq.
\newblock The exit-time problem for a markov jump process.
\newblock {\em The European Physical Journal Special Topics}, 223:3257--3271,
  2014.

\bibitem{Capodaglio2019}
G.~Capodaglio, M.~D'Elia, P.~Bochev, and M.~Gunzburger.
\newblock An energy-based coupling approach to nonlocal interface problems.
\newblock {\em arXiv:2001.03696}, 2019.

\bibitem{Cartis2018}
Coralia Cartis, Lindon Roberts, and Oliver Sheridan-Methven.
\newblock Escaping local minima with derivative-free methods: a numerical
  investigation.

\bibitem{de2017bilevel}
Juan~Carlos De~los Reyes, C-B Sch{\"o}nlieb, and Tuomo Valkonen.
\newblock Bilevel parameter learning for higher-order total variation
  regularisation models.
\newblock {\em Journal of Mathematical Imaging and Vision}, 57(1):1--25, 2017.

\bibitem{de2013image}
Juan~Carlos De~los Reyes and Carola-Bibiane Sch{\"o}nlieb.
\newblock Image denoising: {L}earning the noise model via nonsmooth
  {PDE}-constrained optimization.
\newblock {\em Inverse Problems \& Imaging}, 7(4), 2013.

\bibitem{Delgoshaie2015}
A.H. Delgoshaie, D.W. Meyer, P.~Jenny, and H.~Tchelepi.
\newblock Non-local formulation for multiscale flow in porous media.
\newblock {\em Journal of Hydrology}, 531(1):649--654, 2015.

\bibitem{DElia2017}
M.~D'Elia, Q.~Du, M.~Gunzburger, and R.~Lehoucq.
\newblock Nonlocal convection-diffusion problems on bounded domains and
  finite-range jump processes.
\newblock {\em Computational Methods in Applied Mathematics}, 29:71--103, 2017.

\bibitem{d2014optimal}
Marta D'Elia and Max Gunzburger.
\newblock Optimal distributed control of nonlocal steady diffusion problems.
\newblock {\em SIAM Journal on Control and Optimization}, 52(1):243--273, 2014.

\bibitem{Dener-2019}
Alp Dener and Todd Munson.
\newblock Accelerating {L}imited-{M}emory {Q}uasi-{N}ewton {C}onvergence for
  {L}arge-{S}cale {O}ptimization.
\newblock In Jo{\~a}o M.~F. Rodrigues, Pedro J.~S. Cardoso, J{\^a}nio Monteiro,
  Roberto Lam, Valeria~V. Krzhizhanovskaya, Michael~H. Lees, Jack~J. Dongarra,
  and Peter~M.A. Sloot, editors, {\em Computational Science -- ICCS 2019},
  pages 495--507, Cham, 2019. Springer International Publishing.

\bibitem{Du_12_SIREV}
Q.~Du, M.~D. Gunzburger, R.~B. Lehoucq, and K.~Zhou.
\newblock Analysis and {A}pproximation of {N}onlocal {D}iffusion {P}roblems
  with {V}olume {C}onstraints.
\newblock {\em SIAM Review}, 54(4):667--696, 2012.

\bibitem{d2016identification}
Marta D’Elia and Max Gunzburger.
\newblock Identification of the diffusion parameter in nonlocal steady
  diffusion problems.
\newblock {\em Applied Mathematics \& Optimization}, 73(2):227--249, 2016.

\bibitem{Fife2003}
P.~Fife.
\newblock {\em Some nonclassical trends in parabolic and parabolic-like
  evolutions}, chapter Vehicular Ad Hoc Networks, pages 153--191.
\newblock Springer-Verlag, New York, 2003.

\bibitem{gilboa2007nonlocal}
Guy Gilboa and Stanley Osher.
\newblock Nonlocal linear image regularization and supervised segmentation.
\newblock {\em Multiscale Modeling \& Simulation}, 6(2):595--630, 2007.

\bibitem{gilboa2008nonlocal}
Guy Gilboa and Stanley Osher.
\newblock Nonlocal operators with applications to image processing.
\newblock {\em Multiscale Modeling \& Simulation}, 7(3):1005--1028, 2008.

\bibitem{Ha2011}
Youn~Doh Ha and Florin Bobaru.
\newblock Characteristics of dynamic brittle fracture captured with
  peridynamics.
\newblock {\em Engineering Fracture Mechanics}, 78(6):1156--1168, 2011.

\bibitem{hintermuller2017optimal}
Michael Hinterm{\"u}ller, Carlos~N Rautenberg, Tao Wu, and Andreas Langer.
\newblock Optimal selection of the regularization function in a weighted total
  variation model. part ii: Algorithm, its analysis and numerical tests.
\newblock {\em Journal of Mathematical Imaging and Vision}, 59(3):515--533,
  2017.

\bibitem{hintermuller2015bilevel}
Michael Hinterm{\"u}ller and Tao Wu.
\newblock Bilevel optimization for calibrating point spread functions in blind
  deconvolution.
\newblock 2015.

\bibitem{kunisch2013bilevel}
Karl Kunisch and Thomas Pock.
\newblock A bilevel optimization approach for parameter learning in variational
  models.
\newblock {\em SIAM Journal on Imaging Sciences}, 6(2):938--983, 2013.

\bibitem{Littlewood2010}
D.~Littlewood.
\newblock Simulation of dynamic fracture using peridynamics, finite element
  modeling, and contact.
\newblock In {\em Proceedings of the ASME 2010 International Mechanical
  Engineering Congress and Exposition, Vancouver, British Columbia, Canada},
  2010.

\bibitem{lou2010image}
Yifei Lou, Xiaoqun Zhang, Stanley Osher, and Andrea Bertozzi.
\newblock Image recovery via nonlocal operators.
\newblock {\em Journal of Scientific Computing}, 42(2):185--197, 2010.

\bibitem{Maggiar2018}
Alvaro Maggiar, Andreas Wächter, Irina~S. Dolinskaya, and Jeremy Staum.
\newblock A {D}erivative-{F}ree {T}rust-{R}egion {A}lgorithm for the
  {O}ptimization of {F}unctions {S}moothed via {G}aussian {C}onvolution {U}sing
  {A}daptive {M}ultiple {I}mportance {S}ampling.
\newblock {\em {SIAM} Journal on Optimization}, 28(2):1478--1507, jan 2018.

\bibitem{Meerschaert2012}
M.M. Meerschaert and A.~Sikorskii.
\newblock {\em Stochastic models for fractional calculus}.
\newblock Studies in mathematics, Gruyter, 2012.

\bibitem{MeKl00}
R.~Metzler and J.~Klafter.
\newblock The random walk's guide to anomalous diffusion: a fractional dynamics
  approach.
\newblock {\em Physics Reports}, 339(1):1--77, 2000.

\bibitem{pang2019fpinns}
G.~Pang, L.~Lu, and G.~E. Karniadakis.
\newblock f{PINN}s: Fractional physics-informed neural networks.
\newblock {\em SIAM Journal on Scientific Computing}, 41(4):A2603--A2626, 2019.

\bibitem{Salmon2010}
Joseph Salmon.
\newblock On {T}wo {P}arameters for {D}enoising with {N}on-{L}ocal {M}eans.
\newblock {\em {IEEE} Signal Processing Letters}, 17(3):269--272, mar 2010.

\bibitem{Schekochihin2008}
A.A. Schekochihin, S.C. Cowley, and T.A. Yousef.
\newblock Mhd turbulence: Nonlocal, anisotropic, nonuniversal?
\newblock In {\em In IUTAM Symposium on computational physics and new
  perspectives in turbulence}, pages 347--354. Springer, Dordrecht, 2008.

\bibitem{Schumer2003}
R.~Schumer, D.A. Benson, M.M. Meerschaert, and B.~Baeumer.
\newblock Multiscaling fractional advection-dispersion equations and their
  solutions.
\newblock {\em Water Resources Research}, 39(1):1022--1032, 2003.

\bibitem{Schumer2001}
R.~Schumer, D.A. Benson, M.M. Meerschaert, and S.W. Wheatcraft.
\newblock Eulerian derivation of the fractional advection-dispersion equation.
\newblock {\em Journal of Contaminant Hydrology}, 48:69--88, 2001.

\bibitem{Sharifymoghaddam2015}
Mina Sharifymoghaddam, Soosan Beheshti, Pegah Elahi, and Masoud Hashemi.
\newblock Similarity {V}alidation {B}ased {N}onlocal {M}eans {I}mage
  {D}enoising.
\newblock {\em {IEEE} Signal Processing Letters}, 22(12):2185--2188, dec 2015.

\bibitem{Silling2000}
S.A. Silling.
\newblock Reformulation of elasticity theory for discontinuities and long-range
  forces.
\newblock {\em Journal of the Mechanics and Physics of Solids}, 48:175--209,
  2000.

\bibitem{smith1997susan}
Stephen~M Smith and J~Michael Brady.
\newblock Susan—a new approach to low level image processing.
\newblock {\em International journal of computer vision}, 23(1):45--78, 1997.

\bibitem{tomasi1998bilateral}
Carlo Tomasi and Roberto Manduchi.
\newblock Bilateral filtering for gray and color images.
\newblock In {\em Iccv}, volume~98, page~2, 1998.

\bibitem{troianiello2013elliptic}
Giovanni~Maria Troianiello.
\newblock {\em Elliptic differential equations and obstacle problems}.
\newblock Springer Science \& Business Media, 2013.

\bibitem{troltzsch2010optimal}
Fredi Tr{\"o}ltzsch.
\newblock {\em Optimal control of partial differential equations: theory,
  methods, and applications}, volume 112.
\newblock American Mathematical Soc., 2010.

\bibitem{yaroslavsky1986digital}
Leonid~P Yaroslavsky.
\newblock Digital picture processing: an introduction.
\newblock {\em Applied Optics}, 25:3127, 1986.

\bibitem{Yuan2014}
Gonglin Yuan, Zengxin Wei, and Maojun Zhang.
\newblock An active-set projected trust region algorithm for box constrained
  optimization problems.
\newblock {\em Journal of Systems Science and Complexity}, 28(5):1128--1147,
  nov 2014.

\end{thebibliography}

\end{document}